\documentclass[a4paper]{scrartcl}

\usepackage{amsmath,amssymb,amsthm,amsopn}
\usepackage{hyperref}
\usepackage[capitalize,noabbrev]{cleveref}
\crefname{equation}{Eq.}{Eqs.}
\crefname{subsection}{Subsection}{Subsections}
\usepackage{autonum}
\usepackage{braket}
\usepackage{tikz}
\usetikzlibrary{tikzmark, matrix, decorations.markings, calc, arrows}
\usepackage{tikz-cd}

\theoremstyle{definition}
\newtheorem{theorem}{Theorem}[section]
\newtheorem{definition}[theorem]{Definition}
\newtheorem{lemma}[theorem]{Lemma}
\newtheorem{proposition}[theorem]{Proposition}
\newtheorem{problem}[theorem]{Problem}
\crefname{problem}{Problem}{Problems}
\newtheorem{example}[theorem]{Example}
\newtheorem{corollary}[theorem]{Corollary}
\newtheorem{remark}{Remark}
\renewcommand{\restriction}{\mathord{\upharpoonright}}
\newcommand{\s}{_{*}}
\newcommand{\iso}{\cong}
\newcommand{\intv}{\mathbb{I}}
\renewcommand{\Im}{\mathop{\mathrm{Im}}\nolimits}
\DeclareMathOperator{\Gr}{Gr}
\DeclareMathOperator{\dom}{dom}
\DeclareMathOperator{\rep}{rep}
\DeclareMathOperator{\Ker}{Ker}
\DeclareMathOperator{\Hom}{Hom}
\DeclareMathOperator{\id}{id}
\DeclareMathOperator{\arr}{arr}
\newcommand{\rn}{$\mathbb{R}$}
\newcommand{\rpm}{\mathsf{vect}^\mathbb{R}}

\newcommand{\reltoeq}{\trianglerighteq}
\newcommand{\relto}{\vartriangleright}

\newcommand{\n}{\emptyset}
\renewcommand{\*}{*}

\makeatletter
\def\Int #1{\expandafter\Int@i#1\@nil}
\def\Int@i #1,#2\@nil{{#1{:}#2}}

\def\IntC #1{\expandafter\IntC@i#1\@nil}
\def\IntC@i #1,#2\@nil{{#1{,}#2}}

\def\itoi #1{\expandafter\itoi@i#1\@nil}
\def\itoi@i #1,#2,#3,#4\@nil{_{\Int{#1,#2}}^{\Int{#3,#4}}}
\makeatother

\makeatletter
\def\tikz@delimiter#1#2#3#4#5#6#7#8{%
  \bgroup
    \pgfextra{\let\tikz@save@last@fig@name=\tikz@last@fig@name}%
    node[outer sep=0pt,inner sep=0pt,draw=none,fill=none,anchor=#1,at=(\tikz@last@fig@name.#2),#3]
    {%
      {\nullfont\pgf@process{\pgfpointdiff{\pgfpointanchor{\tikz@last@fig@name}{#4}}{\pgfpointanchor{\tikz@last@fig@name}{#5}}}}%
      \delimitershortfall\z@
      \resizebox*{!}{#8}{$\left#6\vcenter{\hrule height .5#8 depth .5#8 width0pt}\right#7$}%
    }
    \pgfextra{\global\let\tikz@last@fig@name=\tikz@save@last@fig@name}%
  \egroup%
}
\makeatother

\newcommand{\mattikz}[1]{
   \begin{tikzpicture}[
      baseline = (p.center), 
      ampersand replacement=\&,
      decoration={
        markings,
        mark=
        at position 0.5
        with{
          \draw[-] (-2pt,-2pt) -- (2pt,2pt);
          \draw[-] (2pt,-2pt) -- (-2pt,2pt);
        }
      }]
      {#1}
    \end{tikzpicture}
}

\newcommand{\matheader}{
  \matrix[matrix of math nodes, column sep=2mm, row sep=2mm,
  inner sep = 0mm, left delimiter = {[},right delimiter = {]},
  every node/.append style={
    anchor=center,text depth = 0.375em,text
    height=0.875em,minimum width=1.25em}](p)
}

\newcommand{\rowarrow}[2]{
  \draw[->, red, postaction={decorate}] #1  to[out=330,in=30] #2;
}

\newcommand{\colarrow}[2]{
  \draw[->, red, postaction={decorate}] #1 to[out=240,in=300] #2;
}

\title{The persistent homology of a sampled map}
\subtitle{From a viewpoint of quiver representations}
\author{Hiroshi Takeuchi\thanks{
Chubu University. 1200 Matsumoto-cho, Kasugai, Aichi 487-8501, Japan. \newline
e-mail: hiroshi\_takeuchi@isc.chubu.ac.jp \newline
\href{https://orcid.org/0000-0001-8695-1883}{orcid.org/0000-0001-8695-1883}
}}
\date{\today}

\begin{document}
\maketitle
\begin{abstract}
This paper aims to introduce a filtration analysis of sampled maps based on persistent homology, providing a new method for reconstructing the underlying maps.
The key idea is to extend the definition of homology induced maps of correspondences using the framework of quiver representations.
Our definition of homology induced maps is given by most persistent direct summands of representations, and the direct summands uniquely determine a persistent homology.
We provide stability theorems of this process and show that the output persistent homology of the sampled map is the same as that of the underlying map if the sample is dense enough.
Compared to existing methods using eigenspace functors, our filtration analysis has an advantage that no prior information on the eigenvalues of the underlying map is required.
Some numerical examples are given to illustrate the effectiveness of our method.
\end{abstract}
\section{Introduction}
Consider the following problem.
\begin{problem} \label{problem:correspondence}
Let $X$ and $Y$ be topological spaces, and $f\colon X \to Y$ be a continuous map.
If we know only $X$, $Y$, and sampling data $f\restriction_S$, which is a restriction of $f$ on a finite subset $S\subset X$, then can we retrieve any information about the homology induced map $f_*\colon HX \to HY$?
\end{problem}

The map $f\restriction_S$ is called a \emph{sampled map} of $f$.
This paper is motivated by~\cite{correspondence}, which suggests the following analysis for sampled maps.
A \emph{grid $\mathcal{X}$} of $X$ is a finite collection of subsets of $X$ with disjoint interiors such that $\bigcup \mathcal{X} := \bigcup_{X' \in \mathcal{X}} X' = X$.
First, we divide the topological spaces $X$ and $Y$ into grids $\mathcal{X}$ and $\mathcal{Y}$, and let $F$ be the union of regions which have elements of the sample $\Gr (f\restriction_S) := \set{(s, f(s)) | s \in S}$.
That is,
    \begin{equation}
	F:=\set{(x, y) \in X \times Y | \text{$x \in \exists X' \in \mathcal{X}$, $y \in \exists Y' \in \mathcal{Y}$, $(X' \times Y') \cap \Gr (f\restriction_S) \neq \emptyset$}},
\end{equation}
the purple regions in \cref{figure:partition}.
The set $F$ is an approximation of the graph $\Gr(f)$ from the sampled map by this subspace, which is called a correspondence.

\begin{figure}[t]
  \begin{minipage}[t]{0.45\hsize}
  	\centering
  	\includegraphics[width=\hsize]{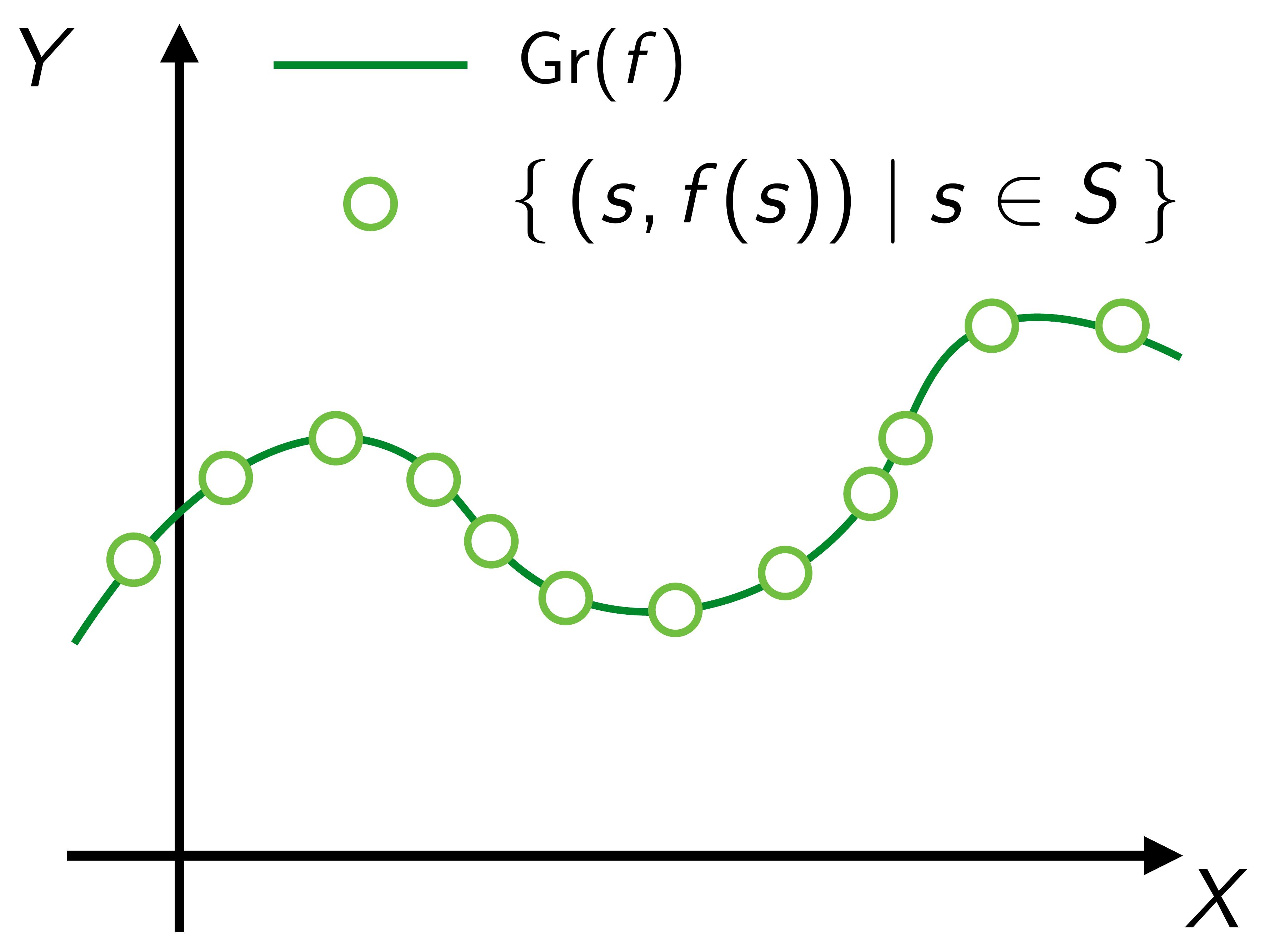}
  	\caption{The graph $\Gr(f)$ of $f$ and the graph of the sampled map.} \label{figure:graph}
  \end{minipage}
  \begin{minipage}[t]{0.05\hsize}
      \hspace{\hsize}
  \end{minipage}
  \begin{minipage}[t]{0.45\hsize}
  	\centering
  	\includegraphics[width=\hsize]{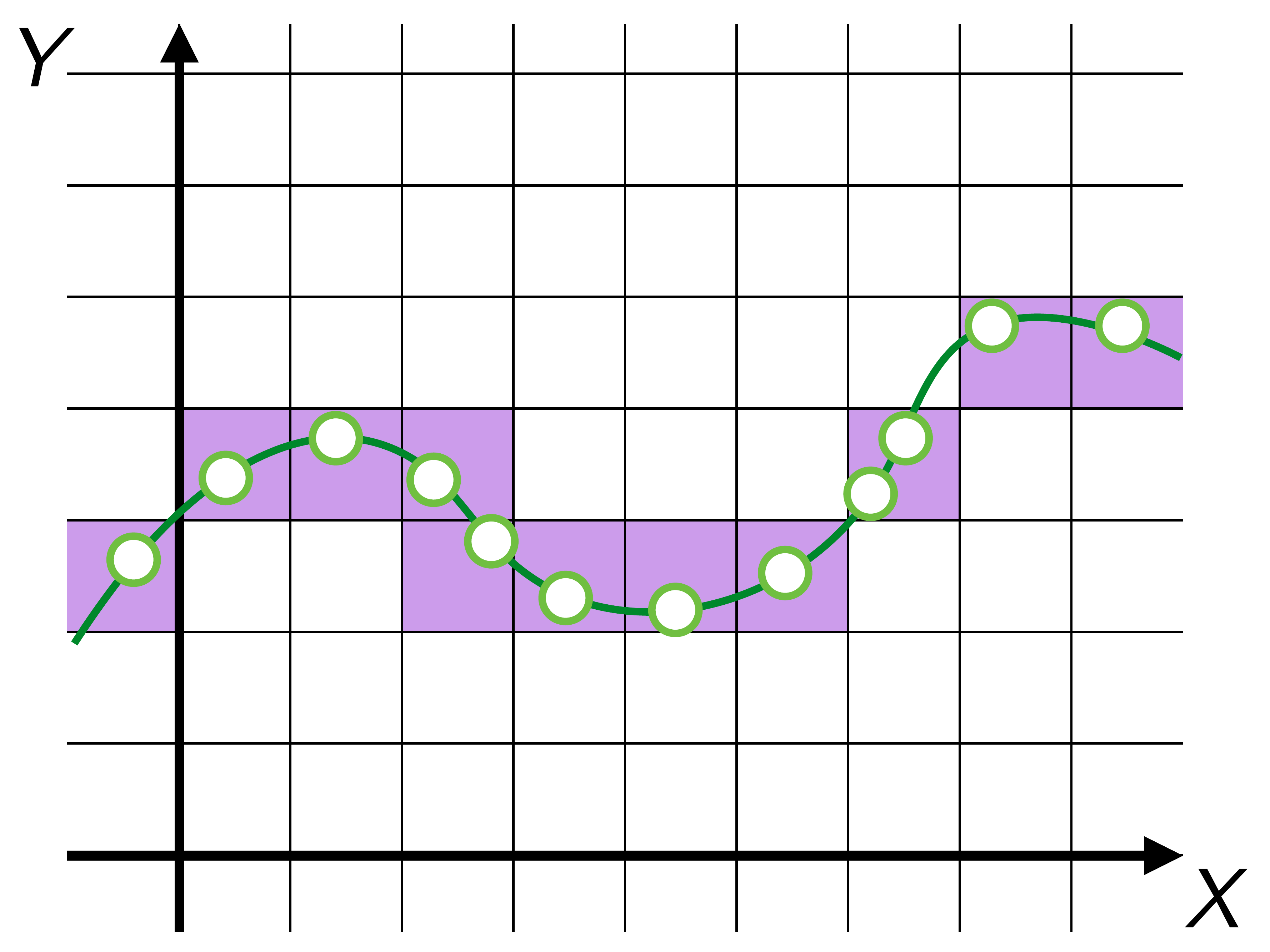}
  	\caption{Both spaces are divided into grids and get a correspondence $F$, which approximates the graph $\Gr(f)$.} \label{figure:partition}
  \end{minipage}
\end{figure}

\begin{definition}
	A \emph{correspondence $F$} from $X$ to $Y$ is a subspace of $X \times Y$.
\end{definition}

\begin{definition} \label{definition:original_induced_map}
	For a correspondence $F$, let $p \colon F \to X$ and $q \colon F \to Y$ be canonical projections, and $p\s \colon HF \to HX$ and $q\s \colon HF \to HY$ be their homology induced maps.
	If $p\s$ and $q\s$ satisfy two properties
	\begin{itemize}
	\item $\Im p\s = HX$ (\emph{homologically complete})
	\item $q\s(\Ker p\s) = 0$ (\emph{homologically consistent}),
	\end{itemize}
	then the \emph{induced map} of $F$ is defined by $F\s := q\s \circ p\s^{-1} \colon HX \to HY$, and is well-defined.
\end{definition}

The graph $\Gr(f)$ of $f$ is a correspondence.
Since $f$ is continuous, $p \colon \Gr(f) \to X$ is homeomorphic.
As a consequence, $p\s$ and $q\s$ for $\Gr(f)$ satisfy homologically completeness and homologically consistency, hence $\Gr(f)\s$ is well-defined.
We remark that this induced map $\Gr(f)\s$ coincides with $f\s$.
The following theorem guarantees that $F\s$ restores $f\s$ when the grid is fine and the sample is dense enough.
\begin{theorem}[{\cite[Theorem 3.10]{correspondence}}]
  \label{theorem:correspondence_graph}
	If a correspondence $F$ satisfies $\Gr(f) \subset F$ and is homologically consistent, then $F\s$ is well-defined and $f\s = F\s$.
\end{theorem}

In \cref{section:induced_maps}, we give a new definition of induced maps of correspondences within the framework of quiver representations.
We will see that the indecomposable decompositions of quiver representations give us an assignment among the bases of $HX$, $HF$, and $HY$, which defines the induced map from $HX$ to $HY$.

The paper~\cite{dynamical_system} gives a way to analyze the eigenspaces of the homology induced map of a self-map (discrete dynamical system).
In this analysis, the authors construct a filtration of simplicial maps from the sampled map and build its persistent homology by applying the homology functor and eigenspace functors.
This construction of the filtration and the new definition of homology induced maps provide a further technique that enables another persistence analysis of sampled maps, shown in \cref{section:persistence_sampled_maps}.
Specifically, the assignment among the bases can compress the three persistent homologies generated by the finite sets $S$, $f\restriction_S$, and $f(S)$, yielding a persistent homology which describes the persistence of the topological mapping from the domain to the image of the induced map $f\s$.
Moreover, we can provide such a persistence analysis also in the above gridded setting, as mentioned in \cref{subsection:grid_filtration}.

The main theorem of this paper is a stability theorem for these processes, \cref{theorem:stability_cech,theorem:stability_grid} in \cref{section:stability_analysis}, which state that these mappings from the input (sampled maps) to output (persistent homology or persistence diagrams) are non-expanding maps.

Finally, we approach 2-D persistence modules in \cref{section:Funtoriality_Other_Intervals} using the above ideas, and show some numerical results in \cref{section:numerical_example}.

\section{Preliminaries}
In this section, we introduce quiver representations and matrix notation for morphisms between $A_n$ type representations.
For more details, the reader can refer to \cite{blue_book} and \cite{matrix_method}, respectively.
\subsection{Quivers and their representations}
Throughout this paper, scalars of vector spaces and coefficient rings of homology groups are a fixed field $K$.
A \emph{quiver} $Q = (Q_0, Q_1, s, t)$ (or simply $(Q_0, Q_1)$) is a directed graph with a set of vertices $Q_0$, a set of arrows $Q_1$, and morphisms $s, t\colon Q_1 \to Q_0$ identifying the source and the target vertex of an arrow.
An arrow $\alpha \in Q_1$ is denoted by $\alpha \colon s(\alpha) \to t(\alpha)$.
A \emph{representation of a quiver $Q$}, denoted $M = (M_a, \varphi_{\alpha})_{a \in Q_0, \alpha \in Q_1}$ (or simply $(M_a, \varphi_{\alpha})$ or $(M,\varphi)$), is a collection of a (finite dimensional) vector space $M_a$ for each vertex $a \in Q_0$ and a linear map $\varphi_{\alpha} \colon M_a \to M_b$ for each arrow $\alpha \colon a \to b \in Q_1$.

A \emph{morphism} from $M = (M_a, \varphi_{\alpha})$ to $M' = (M'_a, \varphi'_{\alpha})$ is defined by 
\begin{equation}
  f:=\set{f_a \colon M_a \to M'_a | a \in Q_0} \colon M \to M'  
\end{equation}
with commutativity
\begin{equation}
  \forall \alpha \colon a \to b \in Q_1,\quad
  \begin{tikzcd}
    M_a \rar["\varphi_\alpha"]  \dar["f_a"] & M_b \dar["f_b"] \\
    M'_a \rar["\varphi'_\alpha"] & M'_b
  \end{tikzcd}.
\end{equation}
The composition of morphisms $f = \set{f_a} \colon M \to M'$ and $g = \set{g_a} \colon M' \to M''$ is $gf = \set{g_a f_a}\colon M \to M''$.

These definitions determine an additive category of representations $\rep(Q)$.
Specifically, $\rep(Q)$ has a zero representation, isomorphisms of representations, and direct sums of representations.
One can see the concrete construction of these in~\cite{blue_book}.

A representation $M$ is \emph{indecomposable} if $M \iso N \oplus N'$ implies $N = 0$ or $N' = 0$.
From the Krull-Remak-Schmidt theorem, every representation $M$ can be uniquely decomposed into a direct sum of indecomposables $M \iso N_1 \oplus \dots \oplus N_s$, unique up to isomorphism and permutations.

A quiver $Q$ is of \emph{finite type} if the number of distinct isomorphism classes of indecomposables is finite, and is of \emph{infinite type} otherwise.

\emph{$A_n(\tau_n)$ type quivers} (or simply \emph{$A_n$ type quivers}) are a class of quivers with the following shape:
\begin{equation}
	A_n(\tau_n):
	\begin{tikzcd}
		\overset{1}{\circ} \rar[leftrightarrow] & \overset{2}{\circ} \rar[leftrightarrow] & \cdots \rar[leftrightarrow] & \overset{n}{\circ}
	\end{tikzcd}
\end{equation}
where $\longleftrightarrow$ denotes a forward arrow $\longrightarrow$ or backward arrow $\longleftarrow$, and $\tau_n$ is a sequence of $n - 1$ symbols $f$ and $b$ which determine the orientation of the arrows.
From Gabriel's theorem~\cite{Gabriel}, every $A_n$ type representation
\begin{equation}
	\begin{tikzcd}
		M : M_1 \rar[leftrightarrow] & M_2 \rar[leftrightarrow] & \cdots \rar[leftrightarrow] & M_n
	\end{tikzcd}
\end{equation}
can be uniquely decomposed into a direct sum of indecomposable \emph{interval representations}
  \begin{align}
	&M \iso \bigoplus_{1 \leq b \leq d \leq n} \intv [b, d]^{m_{b, d}}\quad (m_{b, d} \in \mathbb{Z}_{\geq0}\text{: multiplicity}),\label{eq:decomposition}\\
	\intv [b, d] : 0 \longleftrightarrow &\cdots \longleftrightarrow 0 \longleftrightarrow \overset{b\text{-th}}{K} \overset{\id_K}{\longleftrightarrow} K \overset{\id_K}{\longleftrightarrow} \cdots \overset{\id_K}{\longleftrightarrow} \overset{d\text{-th}}{K} \longleftrightarrow 0 \longleftrightarrow \cdots \longleftrightarrow 0.
\end{align}

In topological data analysis, a central role is played by persistent homology.
The homology of a filtration of simplicial complexes
\begin{equation}
	HX : HX_1 \rightarrow HX_2 \rightarrow \cdots \rightarrow HX_n
\end{equation}
can be regarded as a representation of an $A_n(ff \cdots f)$ type quiver in the framework of quiver representations.
Each interval representation $\intv[b, d]$ corresponds to a generator of a homology group which is born at $HX_b$ and dies at $HX_{d + 1}$, and $d - b$ is called its \emph{lifetime}.
The \emph{persistence diagram} is a multiset determined by the unique decomposition of the persistent homology
\begin{equation}
	D_M = \set{(b, d) | 1 \leq b \leq d \leq n, (b, d) \text{ has multiplicity } m_{b, d}},
\end{equation}
or an image made by plotting this on a plane.
This description allows us to overview the generators of all homology groups, and hence this approach is frequently used for the application of persistent homology.

The framework of quiver representations has extended persistent homology to general representations of quivers.
We call representations of quivers \emph{persistence modules}.
Zigzag persistence modules~\cite{zigzag} are a typical example of the extension, which enable persistence analysis of time series data.

Consider deformations of topological spaces $(X_1, \ldots, X_T)$, a sequence of topological spaces.
The zigzag persistence of the sequence is the representation of an $A_{2T-1}(fbfb\cdots fb)$ type quiver
\begin{equation}
  H(X_1)\to H(X_1 \cup X_2) \gets H(X_2) \to \cdots \to H(X_{T-1} \cup X_T) \gets HX_T
\end{equation}
composed by the unions of neighboring spaces and their canonical inclusions.
The decomposition of the zigzag persistence module as a representation yields a persistence diagram again, where each interval captures the persistence of a homology generator in the deformations of spaces.

\subsection{Matrix notation for morphisms in $\rep(A_n)$}
The paper~\cite{matrix_method} established a new matrix notation for morphisms in $\rep(A_n)$ in the following way.
This notation will give a clear perspective when arguing the well-definedness of persistence analysis in \cref{section:persistence_sampled_maps}.

\begin{definition}[{\cite[Definition 3]{matrix_method}}]
  The \emph{arrow category $\arr(\rep(Q))$} of the category $\rep(Q)$ is a category whose objects are all morphisms in $\rep(Q)$, where morphisms are defined as follows.
  For two objects $f \colon M \to N$ and $f' \colon M' \to N'$ in this category, an morphism from $f$ to $g$ is a pair $(F_M, F_N)$ of morphisms $F_M \colon M \to M'$ and $F_N \colon N \to N'$ satisfying $F_N f = f' F_M$.
  The composition of morphisms $(F_M, F_N)\colon f \to f'$ and $(G_M, G_N) \colon f' \to f''$ is defined by $(G_M F_M, G_N F_N) \colon f \to f''$.
\end{definition}

We remark that every morphism $\varphi \colon V \to W$ between representations of $\rep(A_n)$ is isomorphic to a morphism between direct sums of interval representations
\begin{equation}
  \Phi := \eta_{W}\varphi\eta_V^{-1} \colon \bigoplus_{1 \leq b \leq d \leq n} \intv [b, d]^{m_{b, d}} \to \bigoplus_{1 \leq b \leq d \leq n} \intv [b, d]^{m'_{b, d}}
\end{equation}
in $\arr(\rep(A_n))$, where
\begin{equation}
  \eta_V \colon V \iso \bigoplus_{1 \leq b \leq d \leq n} \intv [b, d]^{m_{b, d}} \text{ and } \eta_W \colon W \iso \bigoplus_{1 \leq b \leq d \leq n} \intv [b, d]^{m'_{b, d}}
\end{equation}
are indecomposable decompositions.
By the following lemma, the morphism $\Phi$ can be written down in a matrix form.
\begin{definition}[{\cite[Definition 4]{matrix_method}}]
  The relation $\reltoeq$ is defined on the set of interval representations of $A_n(\tau_n)$, $\set{\intv[b,d] | 1 \leq b \leq d \leq n}$, by setting $\intv[a,b] \reltoeq \intv[c,d]$ if and only if $\Hom(\intv[a,b],\intv[c,d])$ is nonzero.
  We write $\intv[a,b] \relto \intv[c,d]$ if $\intv[a,b] \reltoeq \intv[c,d]$ and $\intv[a,b] \neq \intv[c,d]$.
\end{definition}
\begin{lemma}[{\cite[Lemma 1]{matrix_method}}]\label{lemma:dimhom}
  Let $\intv[a,b]$ and $\intv[c,d]$ be interval representations of $A_n(\tau_n)$.
  \begin{enumerate}
    \item The dimension of $\Hom(\intv[a,b],\intv[c,d])$ as a $K$-vector space is either 0 or 1.
    \item A $K$-vector space basis $\set{f_{a:b}^{c:d}}$ can be chosen for each $\Hom(\intv[a,b],\intv[c,d])$ such that if $\intv[a,b] \reltoeq \intv[c,d]$, $\intv[c,d] \reltoeq \intv[e,f]$ and $\intv[a,b] \reltoeq \intv[e,f]$, then
    \begin{equation}
      f_{a:b}^{e:f} = f_{c:d}^{e:f}f_{a:b}^{c:d}.
    \end{equation}
  \end{enumerate}
\end{lemma}

We use the notation $[a,b] := \set{a, a+1, \ldots, b}$ over $\mathbb{Z}$.
A candidate for the basis is
\begin{equation}
  (f_{a:b}^{c:d})_i = 
  \begin{cases}
    \id_K &(i \in [a,b] \cap [c,d])\\
    0 &(\text{otherwise}),
  \end{cases}
\end{equation}
in the case that $\intv[a,b] \reltoeq \intv[c,d]$.
In the case that $\intv[a,b] \not\reltoeq \intv[c,d]$, we define $f_{a:b}^{c:d} = 0$ for convenience.
We fix this basis throughout this paper.

The morphism $\Phi$ can be written in block matrix form
\begin{equation}
  \Phi = \left[ \Phi_{a:b}^{c:d} \right],
\end{equation}
where each $\Phi_{a:b}^{c:d} \colon \intv[a,b]^{m_{a,b}} \to \intv[c,d]^{m_{c,d}'}$ is composition of the canonical inclusion $\iota$ with the canonical projection $\pi$:
\begin{equation}
  \begin{tikzcd}
    \intv[a,b]^{m_{a,b}} \rar["\iota"] & \bigoplus\limits_{1 \leq a \leq b \leq n}\intv[a,b]^{m_{a,b}} \rar["\Phi"] & \bigoplus\limits_{1 \leq c \leq d \leq n}\intv[c,d]^{m_{c,d}'} \rar["\pi"] & \intv[c,d]^{m_{c,d}'}.
  \end{tikzcd}
\end{equation}
In a similar way, each block $\Phi_{a:b}^{c:d}$ can also be written in matrix form with the entries in $\Hom(\intv[a,b],\intv[c,d])$:
\begin{equation}
  \Phi_{a:b}^{c:d} = \left[ \phi_j^i \right] \quad (\phi_j^i \in \Hom(\intv[a,b],\intv[c,d]), 1 \leq j \leq m_{a,b}, 1 \leq i \leq m'_{c,d}).
\end{equation}
For each relation $\intv[a,b] \reltoeq \intv[c,d]$, according to \cref{lemma:dimhom} and factoring out $f_{a:b}^{c:d}$ from each $\phi_j^i$, we can get $\phi_j^i = \mu_j^i f_{a:b}^{c:d}$ for some $\mu_j^i \in K$.
In a similar way, factoring out $f_{a:b}^{c:d}$ from each $\Phi_{a:b}^{c:d}$, we get
\begin{equation}
  \Phi_{a:b}^{c:d} = M_{a:b}^{c:d} f_{a:b}^{c:d}
\end{equation}
where each $M_{a:b}^{c:d}$ is a $m_{c,d} \times m_{a,b}$ matrix with the entries in $K$.

\begin{definition}\label{definition:block_matrix_form}
  Let $\varphi$ be a morphism in $\rep(A_n)$. The block matrix form $\Phi(\varphi)$ of $\varphi$ is
  \begin{equation}
    \Phi(\varphi) = \left[ \Phi_{a:b}^{c:d} \right] = \left[ M_{a:b}^{c:d} f_{a:b}^{c:d} \right]
  \end{equation}
\end{definition}

The paper~\cite{matrix_method} shows that isomorphisms in the arrow category $\arr(\rep(A_n))$ correspond to row and column operations in block matrix form.
These operations are performed by matrix multiplication with the same restriction, that is, the block of $\intv[a,b] \not\reltoeq \intv[c,d]$ must be always zero.
The column and row operations are almost the same as that of $K$-matrices, however because of the restriction, addition from a block to certain blocks is not permissible.

Let us discuss column operations.
We define a morphism $\Phi'$ such that the following diagram commutes:
\begin{equation}
  \begin{tikzcd}
    \bigoplus_{1 \leq b \leq d \leq n} \intv [b, d]^{m_{b, d}} \rar["\Phi"] & \bigoplus_{1 \leq b \leq d \leq n} \intv [b, d]^{m'_{b, d}} \\
    \bigoplus_{1 \leq b \leq d \leq n} \intv [b, d]^{m_{b, d}} \uar["\Theta"] \uar[',"\iso"] \urar["\Phi'"] &
  \end{tikzcd}
\end{equation}
where $\Theta$ is an isomorphism.
Namely, $\Phi$ and $\Phi'$ are isomorphic in the arrow category.
The morphism $\Theta$ is also a morphism between direct sums of interval representations, hence in the same way, $\Theta$ can also be written in block matrix form $\left[ C\itoi{a,b,c,d} f\itoi{a,b,c,d} \right]$.
The multiplication $\Phi \Theta$ denotes column operations on $\Phi$, and its block at column $a:b$ and row $c:d$ is
\begin{align}
    \left[ \Phi \Theta \right]\itoi{a,b,c,d} &= \sum_{\intv[a,b] \reltoeq \intv[e,f] \reltoeq \intv[c,d]} (M\itoi{e,f,c,d} f\itoi{e,f,c,d}) (C\itoi{a,b,e,f} f\itoi{a,b,e,f}) \\
   &= (\sum_{\intv[a,b] \reltoeq \intv[e,f] \reltoeq \intv[c,d]} M\itoi{e,f,c,d} C\itoi{a,b,e,f}) f\itoi{a,b,c,d}.
\end{align}

The difference from the usual column operations on $K$-matrices is the part of not only the morphism $f\itoi{a,b,c,d}$ but also the summation $\sum_{\intv[a,b] \reltoeq \intv[e,f] \reltoeq \intv[c,d]}$.
The existence of the morphism $f\itoi{a,b,c,d}$ just means that the block with $\intv[a,b] \not\reltoeq \intv[c,d]$ must always be zero even when added from the other part.
The summation $\sum_{\intv[a,b] \reltoeq \intv[e,f] \reltoeq \intv[c,d]}$ might be a bit more complicated.
It means that not all of the column operations are permissible, but only the following cases.
\begin{itemize}
  \item Elementary column operations (switching, multiplication, and addition) within the same interval is always permissible.
  \item Column addition to $\intv[a,b]$ from another interval $\intv[e,f]$ with relation $\intv[a,b] \relto \intv[e,f]$ is permissible.
\end{itemize}
Similar properties of permissibility hold for row operations.
\begin{itemize}
  \item Elementary row operations (switching, multiplication, and addition) within the same interval is always permissible.
  \item Row addition to $\intv[a,b]$ from another interval $\intv[e,f]$ with relation $\intv[e,f] \relto \intv[a,b]$ is permissible.
\end{itemize}

The following is an example of permissibility in the case of $\arr(\rep(A_3(bf)))$.
\begin{example}\label{example:CLbf}
  The following matrix is the block matrix form of $\arr(\rep(A_3(bf)))$, where we use the symbols $\Int{a,b}$ to denote the rows and columns corresponding to the direct summands $\intv[a,b]^{m_{a,b}}$.
  Each block $M\itoi{a,b,c,d} f\itoi{a,b,c,d}$ is abbreviated to $\*$ if $f\itoi{a,b,c,d} \neq 0$ and $\n$ otherwise.
  The prohibited additions for columns and rows, written as red arrows, correspond to the positions of the zero blocks $\n$.
  Since column addition from lower to upper blocks and row addition from left to right blocks are always prohibited, we write down only the prohibited column additions from upper to lower, and prohibited row addition from right to left.
  A notable fact is that either $f\itoi{a,b,1,3}$ or $f\itoi{1,3,a,b}$ is 0 for arbitrary $(a,b) \neq (1,3)$.
  In other words, $\intv[1,3] \reltoeq \intv[a,b] \reltoeq \intv[1,3]$ if and only if $(a,b) = (1,3)$.
  We will refer to this fact in the proof of \cref{lemma:funtoriality}.
  \begin{equation}
    \mattikz{
      \matheader{
        \* \& \n \& \n \& \n \& \n \& \n \\
        \n \& \* \& \n \& \n \& \n \& \n \\
        \* \& \* \& \* \& \n \& \n \& \n \\
        \* \& \n \& \* \& \* \& \n \& \n \\
        \n \& \* \& \* \& \n \& \* \& \n \\
        \n \& \n \& \* \& \* \& \* \& \* \\
      };
      \foreach[count = \i] \v in {\(\Int{3, 3}\), \(\Int{1, 1}\), \(\Int{1, 3}\), \(\Int{2, 3}\), \(\Int{1, 2}\), \(\Int{2, 2}\)}{
        \node[left=2.5ex] at (p-\i-1) {\scriptsize\v};
      }
  	\foreach[count = \i] \v in {\(\Int{3, 3}\), \(\Int{1, 1}\), \(\Int{1, 3}\), \(\Int{2, 3}\), \(\Int{1, 2}\), \(\Int{2, 2}\)}{
        \node[above=1.5ex] at (p-1-\i) {\scriptsize\v};
  	}
  	\node[right=6pt] (LR) at (p-1-6.east |- p-6-6.south east) {};
  	\rowarrow{(LR |- p-1-6.east)}{(LR |- p-2-6.east)};
  	\rowarrow{(LR |- p-1-6.east)}{(LR |- p-5-6.east)};
  	\rowarrow{(LR |- p-1-6.east)}{(LR |- p-6-6.east)};
  
  	\rowarrow{(LR |- p-2-6.east)}{(LR |- p-4-6.east)};
  	\rowarrow{(LR |- p-2-6.east)}{(LR |- p-6-6.east)};
  	
  	\rowarrow{(LR |- p-4-6.east)}{(LR |- p-5-6.east)};
  
  	\colarrow{(p-6-6.south |- LR)}{(p-6-1.south |- LR)};
  	\colarrow{(p-6-6.south |- LR)}{(p-6-2.south |- LR)};
  	
  	\colarrow{(p-6-5.south |- LR)}{(p-6-4.south |- LR)};
  	\colarrow{(p-6-5.south |- LR)}{(p-6-1.south |- LR)};
  	
  	\colarrow{(p-6-4.south |- LR)}{(p-6-2.south |- LR)};
  
  	\colarrow{(p-6-2.south |- LR)}{(p-6-1.south |- LR)};
    }
  \end{equation}
\end{example}

\section{The induced maps via quiver representations}\label{section:induced_maps}
In this section, we redefine the induced map of a correspondence by using quiver representations.
Let $H = H(-;K)$ be the homology functor with coefficient $K$.
As a representation of an $A_3$ type quiver, the diagram $HX \overset{p\s}{\gets} HF \overset{q\s}{\to} HY$ induced by a correspondence $F \subset X \times Y$ can be decomposed into a direct sum of interval representations:
\begin{equation}
	h \colon (HX \overset{p\s}{\gets} HF \overset{q\s}{\to} HY) \overset{\iso}{\to} \bigoplus_{1\leq b \leq d \leq 3} \intv[b,d]^{m_{b,d}}.
\end{equation}
This indecomposable decomposition can be written as the diagram
\begin{equation}
  \begin{tikzcd}
    & HX \dar[',"h_X"] \dar["\iso"] & \lar[',"p\s"] HF \rar["q\s"] \dar["\iso"] & HY \dar[',"h_Y"] \dar["\iso"] \\
    \bigoplus\limits_{1\leq b \leq d \leq 3} \intv[b,d]^{m_{b,d}} = & K^{\dim HX} & \lar K^{\dim HF} \rar["q^K"] & K^{\dim HY}.
  \end{tikzcd}
\end{equation}
The choice of bases gives us a relationship between bases of $HX$ and $HY$, which can be regarded as a map from $HX$ to $HY$.
For example, an interval representation $\intv[1,2]$ assigns an element of the standard basis of $K^{\dim HX}$ to 0 in $K^{\dim HY}$.
Therefore, non-trivial assignment happens only on the interval representations $\intv[1,3] = (K \overset{\id_K}{\gets} K \overset{\id_K}{\to} K)$.
By regarding the other interval representations as $0$ maps from $HX$ to $HY$, we can define a map $\iota_Y \circ \pi_X \colon K^{\dim HX} \to K^{\dim HY}$ factoring the interval $\intv[1,3]$
\begin{equation}
  \begin{tikzcd}
    \bigoplus\limits_{1\leq b \leq d \leq 3} \intv[b,d]^{m_{b,d}} \dar[twoheadrightarrow] = & K^{\dim HX} \dar[', twoheadrightarrow, "\pi_X"] & \lar K^{\dim HF} \rar["q^K"] \dar[', twoheadrightarrow, "\pi_F"] & K^{\dim HY} \dar[twoheadrightarrow, shift right=1ex] \\
    \intv[1,3]^{m_{1,3}} = & K^{m_{1,3}} & \lar[equal] K^{m_{1,3}} \rar[equal] & K^{m_{1,3}}, \uar[hookrightarrow, shift right=1ex, ', "\iota_Y"]
  \end{tikzcd}
\end{equation}
where the arrows $\twoheadrightarrow$ are the canonical projections of the vector spaces, and the morphism $\iota_Y$ is the canonical injection of the vector space.
Composing the path of morphisms, we can define the \emph{induced map of $F$ through $h$} as
\begin{equation}
  F\s := h_Y^{-1} \circ \iota_Y \circ \pi_X \circ h_X \colon HX \to HY.
\end{equation}

\begin{figure}[t]
		\centering
		\includegraphics[width=0.7\hsize]{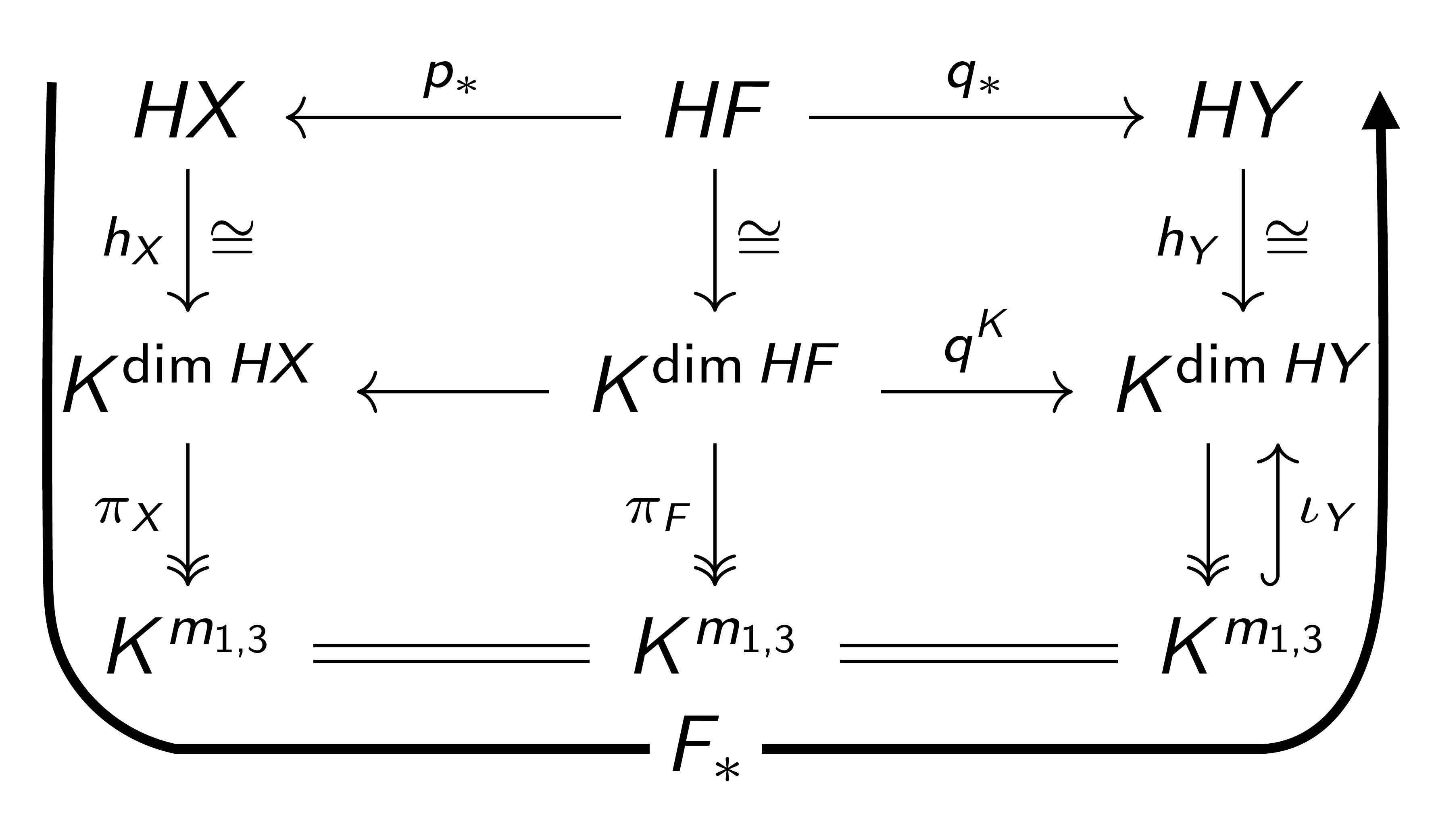}
		\caption{An overview of our definition of the induced map $F\s$ of a correspondence. The isomorphism from the first row to the second row is an indecomposable decomposition. The inclusion map on the right-hand side is the canonical injection of the vector space.} \label{figure:def}
\end{figure}

This definition does not need the two assumptions mentioned in \cref{definition:original_induced_map}.
Although our definition depends on the choice of isomorphism of indecomposable decomposition, when the two assumptions are satisfied, our definition coincides with the original definition $q\s \circ p\s^{-1}$ by the following theorem.

\begin{theorem}
  Define $F\s := h_Y^{-1} \circ \iota_Y \circ \pi_X \circ h_X$.
  If $F$ is homologically complete and homologically consistent, then $F\s = q\s \circ p\s^{-1}$.
\end{theorem}
\begin{proof}
  The claim to prove is
  \begin{equation}
    q\s \circ p\s^{-1} = h_Y^{-1} \circ \iota_Y \circ \pi_X \circ h_X.
  \end{equation}
  Since the morphism $p\s$ is surjective, this is equivalent to
  \begin{equation}
    q\s \circ p\s^{-1} \circ p\s = h_Y^{-1} \circ \iota_Y \circ \pi_X \circ h_X \circ p\s.
  \end{equation}
  In addition, $q\s = q\s \circ p\s^{-1} \circ p\s$ because of the homological consistency $q\s(\Ker(p\s)) = 0$, hence what we should prove is
  \begin{equation}
    q\s = h_Y^{-1} \circ \iota_Y \circ \pi_X \circ h_X \circ p\s.
  \end{equation}
  By chasing the diagram of \cref{figure:def}, this equation results in
  \begin{equation}
    q^K = \iota_Y \circ \pi_F.
  \end{equation}
  The standard basis of $K^{\dim HF}$ corresponds to the standard bases of the four intervals $\intv[2,3]$, $\intv[2,2]$, $\intv[1,2]$, and $\intv[1,3]$.
  Here we remark that the homological consistency $q\s(\Ker(p\s)) = 0$ is equivalent to $m_{2,3} = 0$, namely $\intv[2,3]$ does not exist as a direct summand.
  Moreover, the basis corresponding to $\intv[2,2]$ and $\intv[1,2]$ is mapped to $0$ by both $q^K$ and $\iota_Y \circ \pi_F$.
  By definition it is clear that $q^K(a) = \iota_Y \circ \pi_F(a)$ holds for each element $a$ of the standard basis of $K^{\dim HF}$ corresponding to the standard bases of $\intv[1,3]$.
\end{proof}

\section{Persistence analysis for sampled maps}\label{section:persistence_sampled_maps}
The ability to decompose and focus exclusively on the interval representation $\intv[1,3]$ provides persistence analysis for sampled maps.
Let us consider the following problem, which is similar to \cref{problem:correspondence} but requires additional assumptions of embeddings.
\begin{problem}\label{problem:dynamical}
Let $f\colon X \to Y$ be a continuous map for $X, Y \subset \mathbb{R}^n$.
If $X$, $Y$, and $f$ are unknown, and we know only a sampled map $f\restriction_S \colon S \to f(S)$ which is a restriction of $f$ on a finite subset $S\subset X$, then can we retrieve any information about the homology induced map $f_*\colon HX \to HY$?
\end{problem}
Note that the sampling $S$ is a point cloud capturing topological features of $X$ when $S$ is dense enough.
Originally the paper~\cite{dynamical_system} sets this problem with the more additional assumption that $X = Y$ and constructs a persistent homology of eigenspaces of the sampled map, in order to analyze the eigenspaces of the discrete dynamical system $f$.

In this section, we explain how to construct another persistent homology of a sampled map, which captures the generator of $HX$ and $Hf(X)$ connected by $f$.
We can utilize two types of filtrations in \cref{subsection:construction_using_simplicial_complexes,subsection:grid_filtration}.
The former filtration is generated using simplicial complexes, and we will prove a stability theorem (\cref{theorem:stability_cech}) of this construction in \cref{section:stability_analysis}.
The latter filtration is generated using grids as in \cref{section:induced_maps}, and also derives stability (\cref{theorem:stability_grid}).
The stability theorem for the latter, however, requires more assumptions than for the former.
Hence we explain in this order.

\subsection{Construction using simplicial complexes}\label{subsection:construction_using_simplicial_complexes}
First, we generate a filtration of abstract simplicial complexes of $S$,
\begin{equation}
  C_1 \subset C_2 \subset \cdots \subset C_\ell,
\end{equation}
each simplex of which has elements of $S$ as its vertices, so that the filtration can capture the topology of the underlying space $X$.
For example, \v{C}ech complexes or Vietoris--Rips complexes~\cite{computational_topology_frog_book} are available.
\begin{definition}\label{definition:cech_vietorisrips}
  Let $P \subset \mathbb{R}^n$ be a finite set.
  The \emph{\v{C}ech complex $\Gamma_r$} for $P$ with a radius $r$ is an abstract simplicial complex defined as
  \begin{equation}
    \Gamma_r = \set{\sigma \subset P | \bigcap_{p \in \sigma} B(p;r) \neq \emptyset},
  \end{equation}
  where $B(p;r)$ is the closed ball of center $p$ and radius $r$.
  Let $d_{\mathbb{R}^n}$ be the Euclidean metric on $\mathbb{R}^n$.
  The \emph{Vietoris--Rips complex $V_r$} for $P$ with a radius $r$ is an abstract simplicial complex defined as
  \begin{equation}
    V_r = \set{\sigma \subset P | \forall p_1, p_2 \in \sigma, d_{\mathbb{R}^n}(p_1, p_2) \leq 2r}.
  \end{equation}
\end{definition}

Similarly, we also generate a filtration of abstract simplicial complexes
\begin{equation}
  D_1 \subset D_2 \subset \cdots \subset D_\ell
\end{equation}
for $f(S)$.

Using these filtrations, we attempt to build a filtration of maps from $f\restriction_S$ to analyze the persistence of the original map $f$, in analogy with the classical technique of persistent homology.
Although we expect the sampled map to derive a simplicial map $C_i \to D_i$ on each $i$-th filter, in general, they can derive only a simplicial partial map\footnote{A correspondence $F$ from $X$ to $Y$ is \emph{partial map} if $F(x)$ is a singleton or empty set for all $x \in X$, where $F(x):=\set{y \in Y | (x,y) \in F}$.}
$f_i \colon C_i \nrightarrow D_i$.
Namely, for a simplex $\set{s_1, \ldots, s_d} \in C_i$, it is not assumed that $\set{f(s_1), \ldots, f(s_d)} \in D_i$.
Hence, the conventional technique computing topological persistence for simplicial maps~\cite{persisitence_for_simplicial_maps} is not available in this setup.

We remark that the graph $\Gr(f\restriction_S) = \set{(s,f(s)) \in \mathbb{R}^n \times \mathbb{R}^n | s \in S}$ of the sampled map is a point cloud in $\mathbb{R}^n \times \mathbb{R}^n$.
Let us define the $i$-th abstract simplicial complex $G_i$ of $\Gr(f\restriction_S)$ as
\begin{align}
  G_i := \bigl\{ \set{(s_1,f(s_1)), \ldots, (s_d,f(s_d))} \subset \Gr(f\restriction_S) \mid & \set{s_1, \ldots, s_d} \in C_i, \notag \\
  & \set{f(s_1), \ldots, f(s_d)} \in D_i \bigr\}.\label{eq:def_of_G_i}
\end{align}
One can show that $\set{G_i}$ forms a filtration.
The sequence of the partial maps can be regarded as a sequence $\set{C_i \overset{p_i}{\leftarrow} G_i \overset{q_i}{\to} D_i}$ of pairs of canonical projections $(p_i, q_i)$.
This sequence of pairs constitutes a filtration induced by the inclusion maps of the filtration
\begin{equation}
  \begin{tikzcd}
    & \vdots & \vdots & \vdots \\
    f_{i+1} \colon & C_{i+1} \uar[hook'] & G_{i+1} \lar["p_{i+1}"'] \rar["q_{i+1}"] \uar[hook'] & D_{i+1} \uar[hook'] \\
    f_i \colon & C_i \uar[hook'] & G_i \lar["p_i"'] \rar["q_i"] \uar[hook'] & D_i \uar[hook'] \\
    & \vdots \uar[hook'] & \vdots \uar[hook'] & \vdots \uar[hook']
  \end{tikzcd}.
\end{equation}
Applying the homology functor to the sequence, we obtain a sequence of representations of the $A_3(bf)$ type quiver:
\begin{equation}\label{eq:triple_homology}
  \begin{tikzcd}
    & \vdots & \vdots & \vdots \\
    {f_{i+1}}\s \colon & HC_{i+1} \uar & HG_{i+1} \lar["{p_{i+1}}\s"'] \rar["{q_{i+1}}\s"] \uar & HD_{i+1} \uar \\
    {f_i}\s \colon & HC_i \uar & HG_i \lar["{p_i}\s"'] \rar["{q_i}\s"] \uar & HD_i \uar \\
    & \vdots \uar & \vdots \uar & \vdots \uar
  \end{tikzcd}.
\end{equation}

\begin{remark}\label{remark:comparison_dom_and_E}
  The earlier research~\cite{dynamical_system} used domains of partial maps to construct a similar filtration.
  For a partial map $f_i \colon C_i \nrightarrow D_i$\footnote{To be accurate, the earlier research considers only the case that $C_i = D_i$ to analyze self-maps.}, the domain $\dom f_i$ is defined as
  \begin{equation}
    \dom f_i := \set{\set{s_1, \ldots, s_d} \in C_i | \set{f(s_1), \ldots, f(s_d)} \in D_i}.
  \end{equation}
  Using the inclusion $\iota_i \colon \dom f_i \hookrightarrow C_i$ and the simplicial map $f'_i \colon \dom f_i \to D_i$ induced by the map $f$ on $\dom f_i$, in a similar way the domains induce the representation
  \begin{equation}\label{eq:dom_representation}
    \begin{tikzcd}
      & \vdots & \vdots & \vdots \\
      {f_{i+1}}\s \colon & HC_{i+1} \uar & H\dom f_{i+1} \lar["{\iota_{i+1}}\s"'] \rar["{f'_{i+1}}\s"] \uar & HD_{i+1} \uar \\
      {f_i}\s \colon & HC_i \uar & H\dom f_i \lar["{\iota_i}\s"'] \rar["{f'_i}\s"] \uar & HD_i \uar \\
      & \vdots \uar & \vdots \uar & \vdots \uar
    \end{tikzcd}.
  \end{equation}
  By the definitions of $G_i$ and $\dom f_i$, it is straightforward to see that the induced representation~\eqref{eq:dom_representation} is isomorphic to our representation~\eqref{eq:triple_homology}.
  For the sake of consistency from the viewpoint of graphs and correspondences, we adopt the simplicial complexes $\set{G_i}$.
\end{remark}

Here, decomposing each filter $HC_i \gets HG_i \to HD_i$ as a representation to the intervals $\bigoplus_{1\leq b \leq d \leq 3} \intv[b,d]^{m_{b,d}^i}$, the representation~\eqref{eq:triple_homology} is isomorphic to
\begin{equation}
  \Lambda : \bigoplus_{1\leq b \leq d \leq 3} \intv[b,d]^{m_{b,d}^1} \rightarrow \bigoplus_{1\leq b \leq d \leq 3} \intv[b,d]^{m_{b,d}^2} \rightarrow \cdots \rightarrow \bigoplus_{1\leq b \leq d \leq 3} \intv[b,d]^{m_{b,d}^\ell}.
\end{equation}
Projecting to $\intv[1,3]$ again, we obtain a sequence of subrepresentations
\begin{equation}
 \Lambda [1,3] : \intv[1,3]^{m_{1,3}^1} \rightarrow \intv[1,3]^{m_{1,3}^2} \rightarrow \cdots \rightarrow \intv[1,3]^{m_{1,3}^\ell},
\end{equation}
which is three copies of an $A_\ell$ type representation as we will see later.

We should be careful in the construction of $\Lambda[1,3]$.
We write canonical projections and injections defined by direct sum as
\begin{align}
  \pi_i \colon \bigoplus_{1\leq b \leq d \leq 3} \intv[b,d]^{m_{b,d}^i} \to \intv[1,3]^{m_{1,3}^i} \\
  \iota_i \colon \intv[1,3]^{m_{1,3}^i} \to \bigoplus_{1\leq b \leq d \leq 3} \intv[b,d]^{m_{b,d}^i}
\end{align}
respectively, and the morphisms in $\Lambda$ as
\begin{equation}
  \Phi_i \colon \bigoplus_{1\leq b \leq d \leq 3} \intv[b,d]^{m_{b,d}^i} \to \bigoplus_{1\leq b \leq d \leq 3} \intv[b,d]^{m_{b,d}^{i+1}}.
\end{equation}
Then, the morphism in $\Lambda[1,3]$
\begin{equation}
  {\Phi_i}\itoi{1,3,1,3} \colon \intv[1,3]^{m_{1,3}^i} \to \intv[1,3]^{m_{1,3}^{i+1}}
\end{equation}
is defined by
\begin{equation}
  {\Phi_i}\itoi{1,3,1,3} := \pi_{i+1} \circ \Phi_i \circ \iota_i,
\end{equation}
which is the submatrix at $(\Int{1,3},\Int{1,3})$ in block matrix form of $\Phi_i$.

At a glance, this construction seems natural, but ``$\pi \colon \Lambda \to \Lambda[1,3]$'' is not a morphism in the representation category.
Namely,
\begin{equation}
    \begin{tikzcd}
    \bigoplus_{1\leq b \leq d \leq 3} \intv[b,d]^{m_{b,d}^i} \rar["\Phi_i"]  \dar["\pi_i"] & \bigoplus_{1\leq b \leq d \leq 3} \intv[b,d]^{m_{b,d}^{i+1}} \dar["\pi_{i+1}"] \\
    \intv[1,3]^{m_{1,3}^i} \rar["{\Phi_i}\itoi{1,3,1,3}"] & \intv[1,3]^{m_{1,3}^{i+1}}
  \end{tikzcd}
\end{equation}
does not always commute.
Consequently, the choice of isomorphism of indecomposable decomposition on each filter may make a difference in the output persistence diagram.
In order to make sense of this analysis, the output persistence diagram should be uniquely determined and independent of the choice of isomorphism.
The following theorem guarantees uniqueness and independence.

The restriction to the block $(\Int{1,3},\Int{1,3})$ has the following functoriality.

\begin{lemma}\label{lemma:funtoriality}
  Let
  \begin{equation}
    \Theta \colon \bigoplus_{1\leq b \leq d \leq 3} \intv[b,d]^{m_{b,d}^1} \to \bigoplus_{1\leq b \leq d \leq 3} \intv[b,d]^{m_{b,d}^2}
  \end{equation}
  and
  \begin{equation}
    \Psi \colon \bigoplus_{1\leq b \leq d \leq 3} \intv[b,d]^{m_{b,d}^2} \to \bigoplus_{1\leq b \leq d \leq 3} \intv[b,d]^{m_{b,d}^3} 
  \end{equation}
  be block matrix forms, then
  \begin{equation}
    \left[ \Psi \Theta \right] \itoi{1,3,1,3} = \Psi\itoi{1,3,1,3} \Theta\itoi{1,3,1,3}.
  \end{equation}
\end{lemma}
\begin{proof}
  Let the corresponding scalar matrix symbols of $\Theta$ and $\Psi$ be $M$ and $N$, respectively.
  The block matrix of $\Psi \Theta$ at $(\Int{1,3},\Int{1,3})$ is
  \begin{align}
    \left[ \Psi \Theta \right] \itoi{1,3,1,3} &= \sum_{\intv[1,3] \reltoeq \intv[a,b] \reltoeq \intv[1,3]} (N\itoi{a,b,1,3}f\itoi{a,b,1,3}) (M\itoi{1,3,a,b}f\itoi{1,3,a,b}) \\
    &= (\sum_{\intv[1,3] \reltoeq \intv[a,b] \reltoeq \intv[1,3]} N\itoi{a,b,1,3} M\itoi{1,3,a,b})f\itoi{1,3,1,3}
  \end{align}
  but only $\intv[1,3]$ can be the candidate for the interval $\intv[a,b]$ as we saw in \cref{example:CLbf}. 
  Therefore
  \begin{align}
    \left[ \Psi \Theta \right] \itoi{1,3,1,3} &= N\itoi{1,3,1,3} M\itoi{1,3,1,3} f\itoi{1,3,1,3}\\
    &= (N\itoi{1,3,1,3} f\itoi{1,3,1,3}) (M\itoi{1,3,1,3} f\itoi{1,3,1,3})\\
    &= \Psi\itoi{1,3,1,3} \Theta\itoi{1,3,1,3}.
  \end{align}
\end{proof}

\begin{lemma}\label{lemma:isomorphic_restriction}
  Let $\Theta \colon \bigoplus_{1\leq b \leq d \leq 3} \intv[b,d]^{m_{b,d}^1} \to \bigoplus_{1\leq b \leq d \leq 3} \intv[b,d]^{m_{b,d}^2}$ be the block matrix form of an isomorphism, then $\Theta \itoi{1,3,1,3}$ is an isomorphism.
\end{lemma}
\begin{proof}
  Let $\Psi$ be the inverse of $\Theta$.
  By \cref{lemma:funtoriality},
  \begin{equation}
    \left[ \Psi \Theta \right]\itoi{1,3,1,3} = \Psi\itoi{1,3,1,3} \Theta\itoi{1,3,1,3}
  \end{equation}
  and the left-hand side is the block $(\Int{1,3},\Int{1,3})$ of the identity map, which is the identity map on $\intv[1,3]^{m_{1,3}^1}$.
  Similarly, $\Theta\itoi{1,3,1,3} \Psi\itoi{1,3,1,3}$ is identity map on $\intv[1,3]^{m_{1,3}^2}$, hence $\Theta\itoi{1,3,1,3}$ is isomorphic.
\end{proof}

\begin{theorem}\label{theorem:uniqueness_restriction}
  The isomorphism class of
  \begin{equation}
    {\Phi_i}\itoi{1,3,1,3} \colon \intv[1,3]^{m_{1,3}^i} \to \intv[1,3]^{m_{1,3}^{i+1}}
  \end{equation}
  is uniquely determined and independent of the choice of the bases of
  \begin{equation}
    \Phi_i \colon \bigoplus_{1\leq b \leq d \leq 3} \intv[b,d]^{m_{b,d}^i} \to \bigoplus_{1\leq b \leq d \leq 3} \intv[b,d]^{m_{b,d}^{i+1}}.
  \end{equation}
\end{theorem}

\begin{proof}
  Let $\Psi_i$ be a morphism isomorphic to $\Phi_i$, which is written as a commutative diagram
  \begin{equation}
    \begin{tikzcd}
    \bigoplus_{1 \leq b \leq d \leq n} \intv [b, d]^{m_{b, d}} \rar["\Phi_i"]  \dar[',"C"] \dar["\iso"] & \bigoplus_{1 \leq b \leq d \leq n} \intv [b, d]^{m'_{b, d}} \\
    \bigoplus_{1 \leq b \leq d \leq n} \intv [b, d]^{m_{b, d}} \rar["\Psi_i"] & \bigoplus_{1 \leq b \leq d \leq n} \intv [b, d]^{m'_{b, d}} \uar["R"] \uar[',"\iso"]
    \end{tikzcd}
  \end{equation}
  for some isomorphisms $C$ and $R$.
  Namely, $\Phi_i = R \Psi_i C$, and by \cref{lemma:funtoriality}, applying the restriction yields
  \begin{align}
    {\Phi_i}\itoi{1,3,1,3} &= \left[ R \Psi_i C \right] \itoi{1,3,1,3} \\
    &= R\itoi{1,3,1,3} {\Psi_i}\itoi{1,3,1,3} C\itoi{1,3,1,3}.
  \end{align}
  By \cref{lemma:isomorphic_restriction}, $R\itoi{1,3,1,3}$ and $C\itoi{1,3,1,3}$ are isomorphisms, hence ${\Phi_i}\itoi{1,3,1,3}$ is isomorphic to ${\Psi_i}\itoi{1,3,1,3}$.
\end{proof}

Since regarding $\intv[1,3] = (K \overset{\id_K}{\gets} K \overset{\id_K}{\to} K)$ as $K$ omits no information, the sequence $\Lambda[1,3]$ can be seen as a sequence of vector spaces
\begin{equation}
  K^{m_{1,3}^1} \rightarrow K^{m_{1,3}^2} \rightarrow \cdots \rightarrow K^{m_{1,3}^\ell},
\end{equation}
which is an $A_\ell$ type representation.
We call this representation \emph{persistent homology of the sampled map $f\restriction_S$}.
Decomposing into intervals, we can draw a persistence diagram, which shows us the robustness of the generators of homology in both filtrations, which are assigned by $f$.
Simultaneously we have constructed the filtration of complexes approximating the unknown spaces $X$ and $f(X)$.

\paragraph{In comparison with the earlier research}
This persistence diagram does not provide any information about eigenvectors, unlike~\cite{dynamical_system}; however, it can be widely applied.
First, since our method does not use the eigenspace functor, we need not require both sides' spaces to be the same.
Therefore, even in the case of sampled dynamical systems $X = Y$ like the previous research, we can weaken the assumption $f\restriction_S \colon S \to S$ to $f\restriction_S \colon S \to f(S)$ and take another filtration on $f(S)$.
(If $f(S)$ is not dense enough for sampling $X$, then we can take $S \cup f(S)$ instead.)
Moreover, since the previous method needs to set an eigenvalue before analysis, they have to predict some behavior of $f$ in advance, but our method does not need any prior information.
The numerical experiments in \cref{section:numerical_example} will emphasize this difference.

\subsection{Constriction using a grid}\label{subsection:grid_filtration}
Subsequently, we provide another construction of a filtration constructed by dividing the spaces.
Suppose the spaces $X$ and $Y$ are embedded into Euclidean space $\mathbb{R}^n$, and both $\mathbb{R}^n$ are divided by $n$-dimensional $\varepsilon$-cubes
\begin{equation}
  \set{[a_1\varepsilon, (a_1 + 1)\varepsilon] \times \dots \times [a_n\varepsilon, (a_n + 1)\varepsilon] | a_1, \dots, a_n \in \mathbb{Z}}.
\end{equation}
To distinguish the two divisions, we write this set as $\mathcal{X}_\varepsilon$ for the $X$ side Euclidean space and $\mathcal{Y}_\varepsilon$ for $Y$ side.
Let $f\restriction_S$ be a sampled map of a continuous map $f \colon X \to Y$, and $p$ and $q$ be the canonical projections of $\mathbb{R}^n \times \mathbb{R}^n$ to the $X$ side and $Y$ side Euclidean spaces, respectively.
First, we generate a correspondence
\begin{align}
  F^{f\restriction_S}_\varepsilon := \bigl\{ (x,y) \in \mathbb{R}^n \times \mathbb{R}^n | & x \in \exists X' \in \mathcal{X}_\varepsilon, \notag \\
  & y \in \exists Y' \in \mathcal{Y}_\varepsilon, (X' \times Y') \cap \Gr (f\restriction_S) \neq \emptyset \bigr\}
\end{align}
where $\Gr (f\restriction_S) := \set{(s, f(s)) | s \in S}$ (see \cref{figure:partition}).

In this setup, we use the $L^\infty$ metric $d_\infty((x_i),(y_i)) := \max_i(|x_i - y_i|)$ for both spaces $\mathbb{R}^n$ and $\mathbb{R}^n \times \mathbb{R}^n$.
To construct a filtration along with the grids, let us define the filtration of a correspondence
\begin{equation}
  F_{i\varepsilon} := (F^{f\restriction_S}_\varepsilon)_{i\varepsilon} = \set{r \in \mathbb{R}^n \times \mathbb{R}^n | d_\infty(r,F^{f\restriction_S}_\varepsilon) \leq i\varepsilon} \quad (i \in \mathbb{Z}_{\geq 1})
\end{equation}
and morphisms $p_{i\varepsilon} := p\restriction_{F_{i\varepsilon}}$ and $q_{i\varepsilon} := q\restriction_{F_{i\varepsilon}}$.
Here we restrict $i = 1,\ldots,\ell$ for sufficiently large $\ell$.
Then we have a similar diagram as before,
\begin{equation}
  \begin{tikzcd}
    \vdots & \vdots & \vdots \\
    p(F_{(i+1)\varepsilon}) \uar[hook] & \lar[',"p_{(i+1)\varepsilon}"] F_{(i+1)\varepsilon} \uar[hook] \rar["q_{(i+1)\varepsilon}"] & q(F_{(i+1)\varepsilon}) \uar[hook] \\
    p(F_{i\varepsilon}) \uar[hook] & \lar[',"p_{i\varepsilon}"] F_{i\varepsilon} \uar[hook] \rar["q_{i\varepsilon}"] & q(F_{i\varepsilon}) \uar[hook] \\
    \vdots \uar[hook] & \vdots \uar[hook] & \vdots \uar[hook] \\
  \end{tikzcd}
\end{equation}
allowing us to obtain the filtrations $\set{p(F_{i\varepsilon})}$ and $\set{q(F_{i\varepsilon})}$, capturing the persistent topological features of $X$ and $f(X)$.
Again, the homology functor derives the sequence of morphisms in $\rep(A_3(bf))$, therefore we can execute the same analysis as before, transforming it into block matrix form, restricting to the blocks $(\Int{1,3},\Int{1,3})$, identifying it with a representation of the $A_\ell$ type quiver, and producing a persistence diagram.

\section{Stability}\label{section:stability_analysis}
In order for a tool in topological data analysis to be considered practical, the output persistence diagrams should behave continuously toward input data.
Such a property is known as a stability theorem~\cite{stability_original,stability_proximity} and has been proved for persistence modules on $\mathbb{R}$.

Let $\mathsf{vect}$ be the category of finite dimensional vector spaces, $\mathbb{R}$ be the poset category of real numbers\footnote{For $x, y \in \mathbb{R}$, a morphism $x \to y$ uniquely exists if and only if $x \leq y$.}.
An object of the functor category $\mathsf{vect}^\mathbb{R}$ is also called a persistence module in some papers.
To distinguish it from our definition, we call this an \emph{$\mathbb{R}$-persistence module}.

Specifically, for an $\mathbb{R}$-persistence module $M$, we assign a vector space $M_t$ for $t \in \mathbb{R}$ and a linear map $\varphi_M(s,t) \colon M_s \to M_t$ for $s \leq t \in \mathbb{R}$, where
\begin{equation}
  \varphi_M(t,t) = \id_{M_t} \quad \text{and} \quad \varphi_M(s,t) \circ \varphi_M(r,s) = \varphi_M(r,t)
\end{equation}
for all $r \leq s \leq t \in \mathbb{R}$.
We call the linear maps $\varphi_M(s,t)$ \emph{transition maps}.
A morphism $f \colon M \to N$ of $\mathbb{R}$-persistence modules is a natural transformation, that is a collection of morphisms $\set{f_t \colon M_t \to N_t | t \in \mathbb{R}}$ commuting
\begin{equation}
  \begin{tikzcd}
    M_s \dar["f_s"] \rar["{\varphi_M(s,t)}"] & M_t \dar["f_t"] \\
    N_s \rar["{\varphi_N(s,t)}"] & N_t
  \end{tikzcd}
\end{equation}
for all $s \leq t \in \mathbb{R}$.

We remark that every persistence module can be similarly regarded as a functor from a finite poset category to $\mathsf{vect}$.

The fundamental objects of $\mathbb{R}$-persistence modules are \emph{interval modules $K_I$} for intervals $I \subset \mathbb{R}$, given by $(K_I)_t = K$ for $t \in I$ and $(K_I)_t = 0$ otherwise, and with the morphism corresponding to $s \leq t \in I$ is an identity map.
As is the case with persistent homology, every $\mathbb{R}$-persistence module can be decomposed into a direct sum of interval modules~\cite{interval_decomposition}.

We can define a distance between $\mathbb{R}$-persistence modules, called the interleaving distance.
\begin{definition}
  For $\delta \geq 0$, define the functor $(\cdot)(\delta) \colon \mathsf{vect}^\mathbb{R} \to \mathsf{vect}^\mathbb{R}$, called the \emph{shift functor}, as follows.
  For an \rn-persistence module $M$, $M(\delta)_t := M_{t+\delta}$ and $\varphi_{M(\delta)}(s,t) := \varphi_M(s+\delta, t+\delta)$.
  For a morphism $f$ in $\rpm$, $f(\delta) := f_{t+\delta}$.
\end{definition}

\begin{definition}
  For an \rn-persistence module $M$ and $\delta \geq 0$, the \emph{$\delta$-transition morphism} $\varphi_M(\delta) \colon M \to M(\delta)$ is defined as $\varphi_M(\delta)_t := \varphi_M(t, t+\delta)$.
\end{definition}

\begin{definition}
  \rn-persistence modules $M$ and $N$ are said to be \emph{$\delta$-interleaved} if there exist morphisms $f \colon M \to N(\delta)$ and $g \colon N \to M(\delta)$ such that
  \begin{equation}
    g(\delta) \circ f = \varphi_M(2\delta) \quad \text{and} \quad f(\delta) \circ g = \varphi_N(2\delta).
  \end{equation}
  The \emph{interleaving distance} $d_I \colon \rpm \times \rpm \to [0,\infty]$ is defined by
  \begin{equation}
    d_I(M,N) := \inf_\delta \set{\text{$M$ and $N$ are $\delta$-interleaved}}.
  \end{equation}
\end{definition}

An often used distance between persistence diagrams is the \emph{bottleneck distance}, which is defined by bijections between them.
It is well-known that the interleaving distance of $\mathbb{R}$-persistence modules is equal to the bottleneck distance of their persistence diagrams~\cite{isometry_theorem_original,isometry_theorem_induced_matching}.
Hence, by showing that a distance between input data is greater than the interleaving distance of their \rn-persistence module, we can prove the stability of the persistence diagrams toward input data.

In analogy with~\cite{dynamical_system}, stability theorems for some filtrations also hold on our analysis as follows.
The discrete setting discussed in \cref{section:persistence_sampled_maps} is enough for implementation, but we extend it to a continuous analysis to prove its stability.

Now we use the following filtrations for $S$ and $f(S)$.
Let $d_{\mathbb{R}^n \times \mathbb{R}^n}$ be a distance on $\mathbb{R}^n \times \mathbb{R}^n$ defined by
\begin{equation}
  d_{\mathbb{R}^n \times \mathbb{R}^n}((x_1, y_1), (x_2, y_2)) := \max\{d_{\mathbb{R}^n}(x_1, x_2), d_{\mathbb{R}^n}(y_1, y_2)\},
\end{equation}
where $d_{\mathbb{R}^n}$ is the Euclidean metric on $\mathbb{R}^n$.
For a subset $U$ of $\mathbb{R}^n$, we define a function $d_U \colon \mathbb{R}^n \to \mathbb{R}_{\geq 0}$ to be infimum distance to a point in $U$.
In the same way, we define the function $d_U \colon \mathbb{R}^n \times \mathbb{R}^n \to \mathbb{R}_{\geq 0}$ for a subset $U$ of $\mathbb{R}^n \times \mathbb{R}^n$.
We use the notation $U_r := d_U^{-1}[0,r]$ to denote the sublevel sets.

Let $\mathsf{Top^{(bf)}}$ be the functor category from the $A_3(bf)$ type quiver $(\cdot \gets \cdot \to \cdot )$ as a poset category to the category of topological spaces.
The sublevel sets $S_r$, $f(S)_r$, and $\Gr(f\restriction_S)_r$ constitute the filtration $\set{S_r \gets \Gr(f\restriction_S)_r \to f(S)_r}$ in $\mathsf{Top^{(bf)}}$ with morphisms induced by the inclusions such that the diagram
\begin{equation}\label{eq:inclusion_maps_topbf}
  \begin{tikzcd}
    S_r & \lar \Gr(f\restriction_S)_r \rar & f(S)_r\\
    S_s \uar[hook] & \lar \Gr(f\restriction_S)_s \uar[hook] \rar & f(S)_s \uar[hook]
  \end{tikzcd}
\end{equation}
commutes for every $s \leq r \in \mathbb{R}_{\geq 0}$.

In the same way as in the discrete analysis, applying the homology functor $H$ to the filtration produces $\set{HS_r \gets H\Gr(f\restriction_S)_r \to Hf(S)_r}$, which is a family of objects in the representation category $\rep(A_3(bf))$ with the induced morphisms from Diagram~(\ref{eq:inclusion_maps_topbf}).

\begin{remark}
  We have constructed the different representations from the previous representations using complexes in \cref{subsection:construction_using_simplicial_complexes}, but these are isomorphic if we adopt \v{C}ech complexes.
  It is known by the Nerve Lemma~\cite{nerve_lemma} that, if $U$ is a finite subset in a metric space, then the sublevel set $U_r$ is homotopy equivalent to the \v{C}ech complex of $U$ with radius $r$.
  Therefore, letting $C_r$, $G_r$, and $D_r$ be \v{C}ech complexes with radius $r$ of the finite subsets $S$, $\Gr(f\restriction_S)$, and $f(S)$, respectively, the induced family $\set{HC_r \gets HG_r \to HD_r}$ is isomorphic to the family $\set{HS_r \gets H\Gr(f\restriction_S)_r \to Hf(S)_r}$.
\end{remark}

Since decomposing every representation into intervals is isomorphic in the functor category $\rep(A_3(bf))^{\mathbb{R}}$, the family $\set{HS_r \gets H\Gr(f\restriction_S)_r \to Hf(S)_r}$ is isomorphic to $\set{\bigoplus_{1\leq b \leq d \leq 3} \intv[b,d]^{m_{b,d}^r}}$, and the induced morphisms can be written in block matrix form again.

By \cref{lemma:funtoriality,theorem:uniqueness_restriction}, the family $\set{\intv[1,3]^{m_{1,3}^r}}$ and the induced morphisms are uniquely determined up to isomorphism.
This gives us three copies of the \rn-persistence module $\set{K^{m_{1,3}^r} = K^{m_{1,3}^r} = K^{m_{1,3}^r}}$.
Thus, we obtain an \rn-persistence module $\set{K^{m_{1,3}^r}}$.
We denote this \rn-persistence module of the sampled map $f\restriction_S$ as $M^{f\restriction_S}$ and call it the \emph{\rn-persistence module of the sampled map}.

\begin{remark}
  The construction of an \rn-persistence module using graphs does not require the assumption that $S$ is a finite set.
  Therefore, if we assume that $\dim HX_r$, $\dim H\Gr(f)_r$, and $\dim Hf(X)_r$ are finite for an arbitrary $r$, then the same analysis can be executed on the filtration $\set{X_r \gets \Gr(f)_r \to f(X)_r}$, deriving an \rn-persistence module $M^f$ in the same way.
  We call this the \emph{\rn-persistence module of the map $f$}.
  The output persistence diagram portrays the robustness of the generators of the homology induced map $f\s$.
\end{remark}

After these setups, we can show the following stability theorem.
\begin{theorem}\label{theorem:stability_cech}
  Let $d_H$ be a Hausdorff distance induced by $d_{\mathbb{R}^n \times \mathbb{R}^n}$.
  For two sampled maps $h \colon S \to \mathbb{R}^n$ and $h' \colon S' \to \mathbb{R}^n$, let $M^h$, $M^{h'}$ be the \rn-persistence modules of the sampled maps.
  Then,
  \begin{equation}
    d_I(M^h, M^{h'}) \leq d_H(\Gr(h), \Gr(h')).
  \end{equation}
\end{theorem}
\begin{proof}
  Let $\varepsilon := d_H(\Gr(h), \Gr(h'))$, and $r$ be an arbitrary real number.
  
  By the definition of Hausdorff distance, $\Gr(h)_r \subset \Gr(h')_{r+\varepsilon}$ and $\Gr(h')_r \subset \Gr(h)_{r+\varepsilon}$.
  Moreover, $\varepsilon = d_H(\Gr(h), \Gr(h'))$ implies that $d_H(S, S') \leq \varepsilon$ and $d_H(h(S), h'(S')) \leq \varepsilon$, hence $S_r \subset S'_{r+\varepsilon}$, $S'_r \subset S_{r+\varepsilon}$, $h(S)_r \subset h'(S')_{r+\varepsilon}$, and $h'(S')_r \subset h(S)_{r+\varepsilon}$ as well.
  These inclusions induce the commutative diagrams:
  \begin{equation}
    \begin{tikzcd}
      S'_{r+\varepsilon} & \lar \Gr(h')_{r+\varepsilon} \rar & h'(S')_{r+\varepsilon}\\
      S_r \uar[hook] & \lar \Gr(h)_r \uar[hook] \rar & h(S)_r \uar[hook]
    \end{tikzcd}
  \end{equation}
  and
  \begin{equation}
    \begin{tikzcd}
      S_{r+\varepsilon} & \lar \Gr(h)_{r+\varepsilon} \rar & h(S)_{r+\varepsilon}\\
      S'_r \uar[hook] & \lar \Gr(h')_r \uar[hook] \rar & h'(S')_r \uar[hook]
    \end{tikzcd}.
  \end{equation}
  By those functoriality, it is straightforward that these inclusions induce $\varepsilon$-interleaving morphisms between $M^h$ and $M^{h'}$, and we have $d_I(M^h, M^{h'}) \leq \varepsilon$.
\end{proof}
Accordingly, the obtained persistence modules and persistence diagrams can only have as much noise as $S$ or its evaluation by $f$.

Since the proof does not use the assumption that $S$ is a finite set, a similar inequality holds for \rn-persistence modules of maps.
\begin{corollary}\label{corollary:stability_map}
  Let $U$ and $U'$ be subsets in $\mathbb{R}^n$.
  If \rn-persistence modules $M^h$ and $M^{h'}$ of maps $h \colon U \to \mathbb{R}^n$ and $h' \colon U' \to \mathbb{R}^n$ are defined, then
  \begin{equation}
    d_I(M^h, M^{h'}) \leq d_H(\Gr(h), \Gr(h')).
  \end{equation}
\end{corollary}
By \cref{corollary:stability_map}, the error (bottleneck distance) between the persistence diagram of a sampled map $f\restriction_S$ and that of the original map $f$ is bounded by the error $d_H(\Gr(f), \Gr(f\restriction_S))$ of the sampled map.
If the sampled map is dense enough, we can infer the persistent generators of $f\s$ from the persistence module of the sampled map.

Finally, let us provide the stability of the persistence analysis using grids.
We regard the persistent homology constructed using a grid as an \rn-persistence module by the embedding induced by
\begin{equation}
  F_r :=
  \begin{cases}
    \emptyset & \text{($r < i\varepsilon$)} \\
    F_{i\varepsilon} & \text{($i\varepsilon \leq r < (i+1)\varepsilon$ for $i \in \mathbb{Z}_{\geq 1}$)}
  \end{cases}.
\end{equation}
\begin{theorem}\label{theorem:stability_grid}
  Let $d_H$ be a Hausdorff distance induced by $d_\infty$.
  For two sampled maps $h \colon S \to \mathbb{R}^n$ and $h' \colon S' \to \mathbb{R}^n$, we write the filtrations of correspondences as $\set{F_r}$ and $\set{F'_r}$, and let $M^h$ and $M^{h'}$ be their output \rn-persistence modules.
  If $d_H (\Gr(h),\Gr(h')) \leq \varepsilon$, then $d_I(M^h, M^{h'}) \leq \varepsilon$.
\end{theorem}
\begin{proof}
  The assumption $d_H (\Gr(h),\Gr(h')) \leq \varepsilon$ derives the inequality $d_H (F_\varepsilon^h ,F_\varepsilon^{h'}) \leq \varepsilon$.
  Therefore there exist the inclusions
  \begin{equation}
    \begin{tikzcd}
      p(F'_{(i+1)\varepsilon}) & \lar F'_{(i+1)\varepsilon} \rar & q(F'_{(i+1)\varepsilon})\\
      p(F_{i\varepsilon}) \uar[hook] & \lar F_{i\varepsilon} \uar[hook] \rar & q(F_{i\varepsilon}) \uar[hook]
    \end{tikzcd}
  \end{equation}
  and
  \begin{equation}
    \begin{tikzcd}
      p(F_{(i+1)\varepsilon}) & \lar F_{(i+1)\varepsilon} \rar & q(F_{(i+1)\varepsilon})\\
      p(F'_{i\varepsilon}) \uar[hook] & \lar F'_{i\varepsilon} \uar[hook] \rar & q(F'_{i\varepsilon}) \uar[hook]
    \end{tikzcd}.
  \end{equation}
  It is clear that these inclusions induce the $\varepsilon$-interleaving morphisms between $M^h$ and $M^{h'}$.
\end{proof}

\section{Application of functoriality to 2-D persistence modules}\label{section:Funtoriality_Other_Intervals}
The functoriality lemma, \cref{lemma:funtoriality}, can be generalized for the restriction to every ``diagonal'' block.
Precisely, since the candidate of intervals $\intv[c,d]$ satisfying relations $\intv[a,b] \reltoeq \intv[c,d] \reltoeq \intv[a,b]$ is only $\intv[c,d] = \intv[a,b]$, 
\begin{equation}
    \left[ \Psi \Theta \right] \itoi{a,b,a,b} = \Psi\itoi{a,b,a,b} \Theta\itoi{a,b,a,b}.
\end{equation}
holds for all $\intv[a,b]$.
This result can be checked easily, not only on the orientation $bf$ but also on every orientation of any length, as follows.

Suppose $\intv[c,d] \neq \intv[a,b]$, which can happen when $a \neq c$ or when $b \neq d$.
In the case that $a \neq c$, we may assume $a < c$ without loss of generality.
When $(c-1)$-th orientation is $f$, consider $g = \set{g_i}_{i=1}^n \in \Hom(\intv[a,b], \intv[c,d])$.
Then, the commutative diagram of the morphism $g$ from $(c-1)$ to $c$ is
\begin{equation}
  \begin{tikzcd}
      0 \rar[rightarrow, "\id_K"] & K \\
      K \rar[rightarrow, "\id_K"] \uar[rightarrow, "g_{c-1}"] & K \uar[rightarrow, swap, "g_c"]
  \end{tikzcd}.
\end{equation}
It is obvious that $g_{c-1} = 0$, and the commutativity derives $g_c = 0$.
Since the commutativity of the diagram on $g$ derives $g_i = 0$ for the other vertices $i$, $g = 0$.
Consequently, $\Hom(\intv[a,b], \intv[c,d]) = 0$, hence $\intv[a,b] \not\reltoeq \intv[c,d]$.
We can show $\intv[c,d] \not\reltoeq \intv[a,b]$ when $(c-1)$-th orientation is $b$ in a similar discussion, using the commutative diagram
\begin{equation}
  \begin{tikzcd}
    K \rar[leftarrow, "\id_K"] & K \\
    0 \rar[leftarrow, "\id_K"] \uar[rightarrow, "g_{c-1}"] & K \uar[rightarrow, swap, "g_c"]
  \end{tikzcd}.
\end{equation}
Similar arguments also hold in the case that $b \neq d$, concluding $\intv[a,b] \not\reltoeq \intv[c,d]$ or $\intv[c,d] \not\reltoeq \intv[a,b]$.

That is why we can extend the statement on the orientation $bf$ to general $\tau_n$ as follows.
\begin{proposition}\label{proposition:general_funtoriality}
  Let
  \begin{equation}
    \Theta \colon \bigoplus_{1\leq a \leq b \leq n} \intv[a,b]^{m_{a,b}^1} \to \bigoplus_{1\leq a \leq b \leq n} \intv[a,b]^{m_{a,b}^2}
  \end{equation}
  and
  \begin{equation}
    \Psi \colon \bigoplus_{1\leq a \leq b \leq n} \intv[a,b]^{m_{a,b}^2} \to \bigoplus_{1\leq a \leq b \leq n} \intv[a,b]^{m_{a,b}^3} 
  \end{equation}
  be block matrix forms of objects in the arrow category $\arr(\rep(A_n(\tau_n)))$, then
  \begin{equation}
    \left[ \Psi \Theta \right] \itoi{a,b,a,b} = \Psi\itoi{a,b,a,b} \Theta\itoi{a,b,a,b}
  \end{equation}
  for all $1 \leq a \leq b \leq n$.
\end{proposition}

This property can be utilized for \emph{2-D persistence modules}, which are representations with the shape
\begin{equation}\label{eq:2Dpersistence}
  \begin{tikzcd}
    M_{n_1, 1} \rar[leftrightarrow] & M_{n_1, 2} \rar[leftrightarrow] & \cdots \rar[leftrightarrow] & M_{n_1, n_2} \\
    \vdots \uar & \vdots \uar & & \vdots \uar \\
    M_{2, 1} \uar \rar[leftrightarrow] & M_{2, 2} \uar \rar[leftrightarrow] & \cdots \rar[leftrightarrow] & M_{2, n_2} \uar \\
    M_{1, 1} \uar \rar[leftrightarrow] & M_{1, 2} \uar \rar[leftrightarrow] & \cdots \rar[leftrightarrow] & M_{1, n_2} \uar 
  \end{tikzcd}
\end{equation}
where every row has the same orientation $\tau_{n_2}$.
The 2-D persistence modules sometimes appear and cause problems in the context of persistence analysis for time series data.
See~\cite{multidimensional_persistence} for details and higher dimensional persistence.

In our context, the 2-D persistence module naturally appears when we consider iterations of a sampled map or compositions of sampled maps.
Suppose we have a time series of some point clouds $\set{S_1, S_2, \dots, S_T}$ in the same Euclidean space, with their transition as maps $\set{f_i \colon S_i \to S_{i + 1}}$.
We generate a filtration of abstract simplicial complexes $C_1^t \subset \dots \subset C_n^t$ for each $S_t$.
As we saw in \cref{section:persistence_sampled_maps}, the maps between points induce a filtration of partial maps $f_i^t \colon C_i^t \nrightarrow C_i^{t + 1}$, which induces a commutative diagram
\begin{equation}\label{eq:iteration_diagram_of_sampled_map}
  \begin{tikzcd}
    \vdots & \vdots & \vdots & & \vdots \\
    C_{i+1}^1 \uar[hook'] & G_{i+1}^1 \lar \rar \uar[hook'] & C_{i+1}^2 \uar[hook'] & \cdots \lar \rar & C_{i+1}^T \uar[hook'] \\
    C_i^1 \uar[hook'] & G_i^1 \lar \rar \uar[hook'] & C_i^2 \uar[hook'] & \cdots \lar \rar & C_i^T \uar[hook'] \\
    \vdots \uar[hook'] & \vdots \uar[hook'] & \vdots \uar[hook'] & & \vdots \uar[hook']
  \end{tikzcd}
\end{equation}
by taking the $i$-th simplicial complex $G_i^t$ of $\Gr(f_i)$ defined as \cref{eq:def_of_G_i}.
As a consequence, the homology functor induces the 2-D persistence module from the above diagram.
In this case, we can observe the 2-D persistence module from the viewpoint that the horizontal (vertical) direction on the diagram describes persistence in time (space, respectively).

Let us go back to Diagram~(\ref{eq:2Dpersistence}).
In the same way as in the specific case $\tau_{n_2} = bf$, Diagram~(\ref{eq:2Dpersistence}) can be regarded as a sequence of morphisms in the category $\rep(A_{n_2}(\tau_{n_2}))$.
By decomposing representations of $A_{n_2}(\tau_{n_2})$ in each row, we can deal with the morphisms as matrices in block matrix form.
Restricting each matrix to the diagonal block $(\Int{a, b},\Int{a, b})$ derives a sequence of matrices whose domains and codomains are direct sums of $\intv[a, b]$.
As this sequence is $b-a$ copies of nonzero representations of $A_{n_1}$ and $n_2-(b-a)$ copies of zero representations, we can take one of the nonzero representations.
Finally, we obtain the persistence diagram by decomposing.
\Cref{proposition:general_funtoriality} ensures the uniqueness of the output persistence diagram.

In the case of 2-D persistence modules derived from Diagram~(\ref{eq:iteration_diagram_of_sampled_map}), when we take the block $(\Int{a, b},\Int{a, b})$ as $(\Int{1,n_2},\Int{1,n_2})$, each generator of the output persistence module survives under all transitions, and its lifetime in the persistence diagram shows how robust it is in the Euclidean space.
Although this process ignores much information stored in the other blocks, it is an approach to 2-D persistence analysis that can capture the rough topological structures.

\section{Numerical experiments}\label{section:numerical_example}
The author has implemented the persistence analysis in \cref{subsection:construction_using_simplicial_complexes}.
Here we fix the field for the coefficient of matrices and the homology functor as $\mathbb{Z}/1009\mathbb{Z}$.
\begin{remark}
  To implement the persistence analysis on computers, we must use finite fields as the coefficient.
  Here, every map is written as a matrix.
  If we choose the field $\mathbb{Z}/2\mathbb{Z}$ as the coefficient, then every entry with the prime factor $2$ in the matrix is regarded as $0$.
  For example, a homology generator $a$ mapped as $f_*(a) = 2a$, such as the example discussed later and in \cref{figure:circle100_2_018}, is ignored.
  Therefore, it is better to choose a larger prime number $p$ as the coefficient $\mathbb{Z}/p\mathbb{Z}$ to retrieve more generators.
\end{remark}
The implementation uses Vietoris--Rips complexes for simplicity, while \v{C}ech complexes are theoretically more satisfying (see~\cite[Section III.2]{computational_topology_frog_book} for details).
Except for the construction of the persistence module from the sequence of the pairs of the maps $\set{({p_i}\s, {q_i}\s)}$, it basically follows the algorithm in~\cite{dynamical_system} (recall \cref{remark:comparison_dom_and_E}).

First, we generate the boundary matrix induced by the filtration of the Vietoris--Rips complexes for each point cloud $S$ and $f(S)$, and then the boundary matrix of the filtration $\set{G_i}$.
We can use the original persistence algorithm~\cite{computational_topology_frog_book} to compute the reduced boundary matrices and the bases of the persistent homology of the filtration.

Second, since we can obtain the homology bases for each filtration, we generate the maps ${p_i}\s$ and ${q_i}\s$ as matrices between the homology basis for each filter.
In the same way, we compute the induced maps of the inclusions $j\s \colon HG_i \to HG_{i+1}$ as matrices.
To obtain the basis of $\intv[1,3]$ for each matrix ${p_i}\s$ and ${q_i}\s$, we execute the following elementary row and column operations.
\begin{enumerate}
  \item Transform ${p_i}\s$ to Smith normal form:
  \begin{equation}
    P_1 {p_i}\s Q_1 = 
    \left[
    \begin{array}{cc}
      I_{r_1} & 0 \\
      0 & 0
    \end{array}
    \right],
  \end{equation}
  where $P_1$ and $Q_1$ are regular matrices corresponding to elementary operations, and $r_1$ is the rank of ${p_i}\s$.
  \item Since ${p_i}\s$ and ${q_i}\s$ share the same basis for the columns, the elementary column operations $Q_1$ are simultaneously performed on ${q_i}\s$
  \begin{equation}
    {q_i}\s Q_1 = 
    \left[
    \begin{array}{cc}
      X_1 & X_2
    \end{array}
    \right],
  \end{equation}
  where $X_1$ is the submatrix on the basis of columns corresponding to the above $I_{r_1}$, and $X_2$ is the submatrix corresponding to the $0$ columns in $P_1 {p_i}\s Q_1$.
  \item Transform $X_2$ to Smith normal form with elementary row operations $P_2$ and elementary column operations $Q_2$
  \begin{equation}
    \left[
    \begin{array}{cc}
      P_2 X_1 & \begin{array}{cc} I_{r_2} & 0 \\ 0 & 0 \end{array}
    \end{array}
    \right]
    =
    \left[
    \begin{array}{ccc}
      X_3 & I_{r_2} & 0 \\
      X_4 & 0 & 0
    \end{array}
    \right]
  \end{equation}
  where $P_2 X_1$ is divided into submatrices of appropriate sizes on the right-hand side.
  We remark that these column operations have no side effect on the ${p_i}\s$ side matrix since every column corresponding to the basis is zero.
  \item Zero out $X_3$ by $I_{r_2}$ using column operations without any side effect:
  \begin{equation}
    \left[
    \begin{array}{ccc}
      0 & I_{r_2} & 0 \\
      X_4 & 0 & 0
    \end{array}
    \right].
  \end{equation}
  \item Transform $X_4$ to Smith normal form with elementary row operations $P_4$ and elementary column operations $Q_4$:
  \begin{equation}
    \left[
    \begin{array}{ccc}
      0 & I_{r_2} & 0 \\
      \begin{array}{cc} I_{r_3} & 0 \\ 0 & 0 \end{array} & 0 & 0
    \end{array}
    \right].
  \end{equation}
  Here the column operations have side effect on $I_{r_1}$ transforming it to a matrix $Q_4$, but $Q_4$ is regular, hence we can transform it to $I_{r_1}$ again using only row operations.
  \item Finally, we obtain the matrix transformations
  \begin{equation}
    {p_i}\s \mapsto 
    \left[
    \begin{array}{cccc}
      I_{r_3} & 0 & 0 & 0 \\
      0 & I_{r_3 - r_1} & 0 & 0
    \end{array}
    \right]
    \text{ and }
    {q_i}\s \mapsto 
    \left[
    \begin{array}{cccc}
      0 & 0 & I_{r_2} & 0 \\
      I_{r_3} & 0 & 0 & 0 \\
      0 & 0 & 0 & 0
    \end{array}
    \right],
  \end{equation}
  which are decomposed into intervals.
  The rows and columns corresponding to two $I_{r_3}$ are pairs of identity maps, which are $\intv[1,3]$.
  Therefore the basis in the columns of $I_{r_3}$ is what we want.
\end{enumerate}

Applying the change of basis of $HG_i$ during the above column operations and restricting to the basis corresponding to $I_{r_3}$, we finally obtain the persistent homology of the sampled map.
At last, we decompose the persistent homology into intervals using the decomposition algorithm in~\cite[Subsection 3.4]{dynamical_system} and plot a persistence diagram.

By tracking the inverse of the matrix transformations executed above, we can write down the cycles corresponding to the generators in the persistence diagram.
As we remarked in \cref{section:induced_maps}, the cycles depend on the choice of the bases of $\intv[1,3]$.
Nevertheless, in the following numerical experiments, we succeed in reconstructing underlying maps.

\subsection{Twice mapping on a circle}
As an example for input data, let us consider the twice map on the unit circle $f \colon S^1 \to S^1$ defined by $f(z) := z^2$.
Note that, in this case, the spaces in \cref{problem:dynamical} are given as $X = Y = S^1$ embedded in $\mathbb{R}^2$, and we regard $\mathbb{R}^2$ as $\mathbb{C}$.
The sampled points of the unit circle are 100 points $z_j := \cos(2\pi\frac{j}{100}) + \sqrt{-1}\sin(2\pi\frac{j}{100})$ for $0 \leq j < 100$, with added Gaussian noise with $\sigma \in [0.00, 0.30]$.

\begin{figure}[t]
		\centering
		\includegraphics{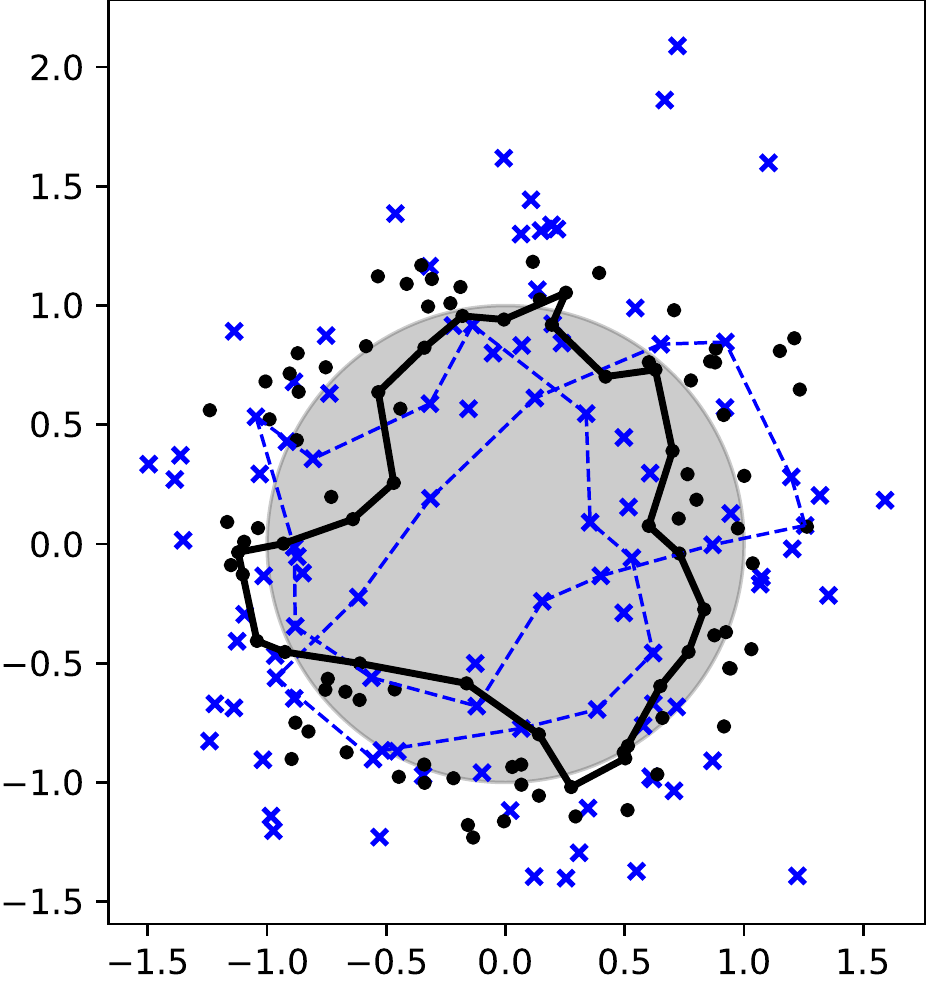}
		\caption{A computational result for $f(z) = z^2$. The number of points is $100$ and the Gaussian noise is at $\sigma = 0.18$. The black points are sampled points for the domain and the blue crosses are its image by $f$. As presented in \cref{figure:pd100_2}, the generator is unique. The corresponding generator in the domain side is described by the black edges approximating the unit circle. The generator in the image side is the blue dashed edges, and we can observe that it turns around the origin twice.} \label{figure:circle100_2_018}
\end{figure}
\begin{figure}[h]
		\centering
		\includegraphics[width=0.75\hsize]{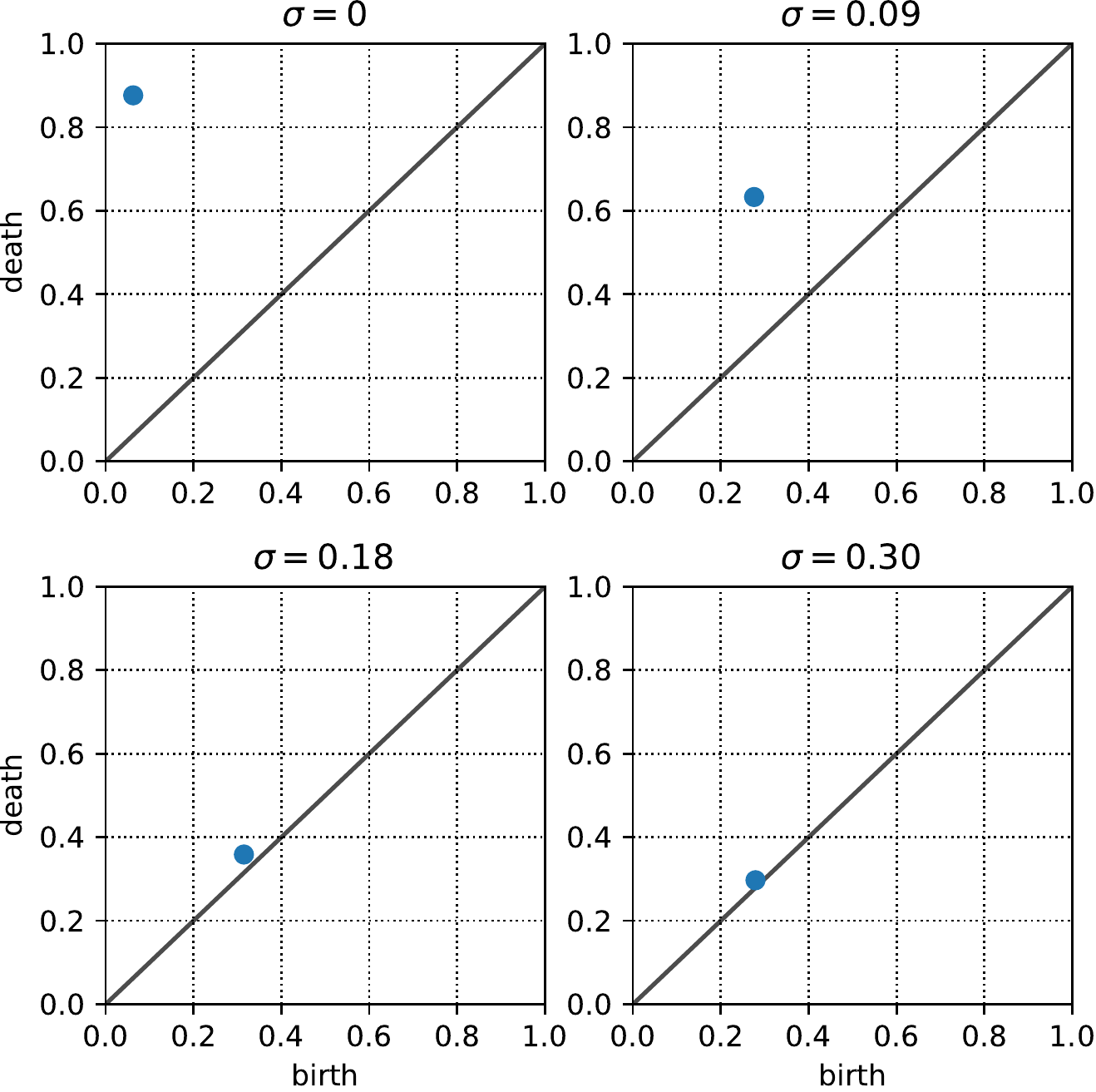}
		\caption{The persistence diagrams for $f(z) = z^2$ with 100 points at $\sigma = 0.00$, $0.09$, $0.18$, and $0.30$. We can observe that each persistence diagram has the unique point, and it goes to the diagonal line as the noise increases.} \label{figure:pd100_2}
		\includegraphics[width=0.75\hsize]{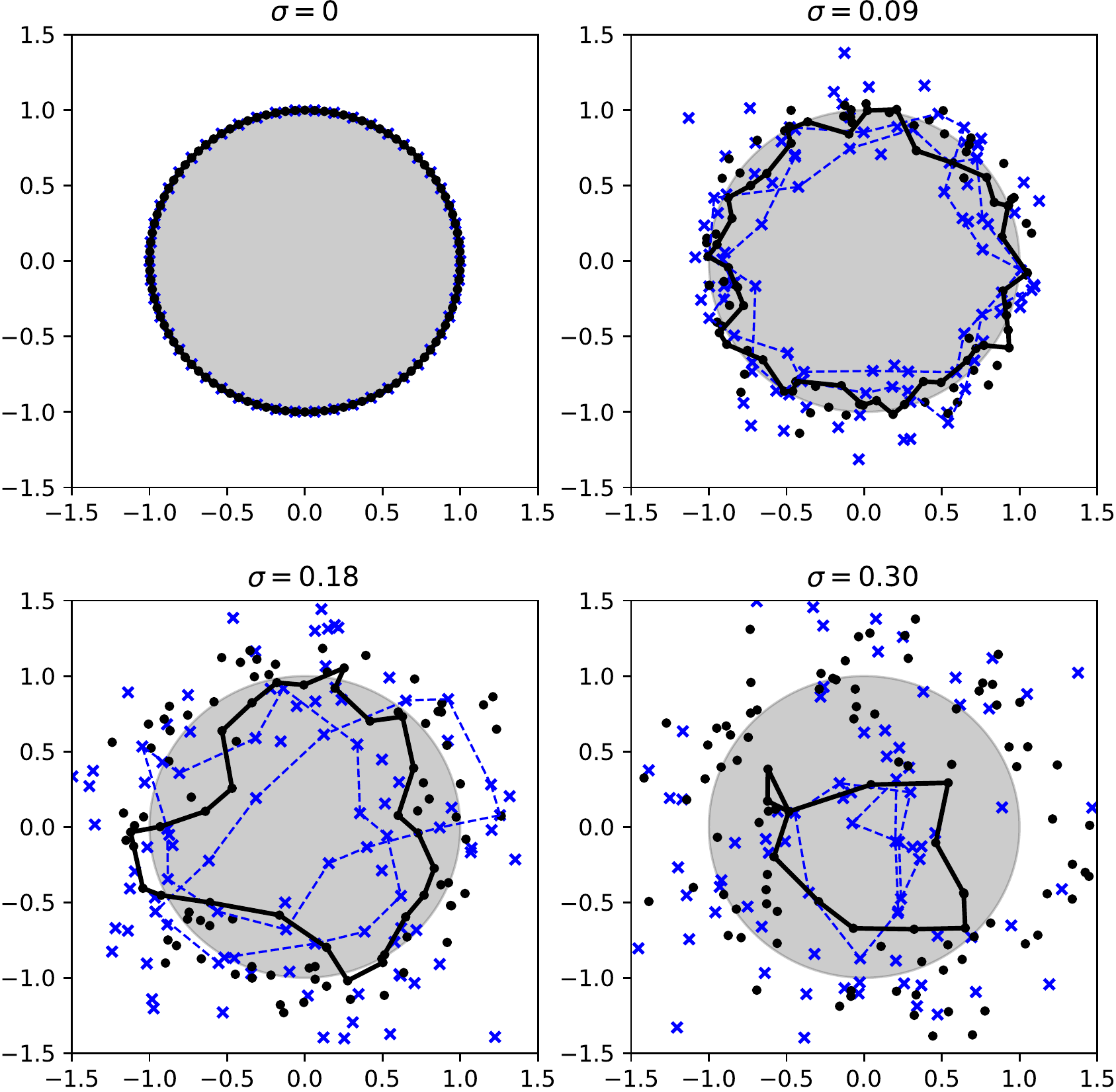}
		\caption{Computational results for $f(z) = z^2$ with 100 points at $\sigma = 0.00$, $0.09$, $0.18$, and $0.30$. Every plotting range is restricted to $[-1.5, 1.5] \times [-1.5, 1.5]$ to observe the generators. The corresponding persistence diagrams are in \cref{figure:pd100_2}.} \label{figure:circle100_2_all}
\end{figure}

A computational result is presented in \cref{figure:circle100_2_018}, which portrays the sampled map at $\sigma = 0.18$ and its unique generator of the persistence diagram.
The generator is the corresponding cycle in $HG_{b}$ at the birth radius $b$ and is indeed approximating the unit circle, and we can see that its image is turning around the origin twice.
Results for other noises are shown in \cref{figure:circle100_2_all}.

\Cref{figure:pd100_2} presents the persistence diagrams under changing $\sigma$ from $0$ to $0.3$.
As expected, the lifetime of the unique generator decreases as the noise increases.

\subsection{Inverse mapping on a circle}
To emphasize the difference from the existing method using eigenspace functors, let us consider the inverse map on the unit circle $g \colon S^1 \to S^1$ defined by $f(z) := z^{-1}$.
The sampled points $\set{z_j}$ and the range of the parameter $\sigma$ of Gaussian noises are the same as before.
The computational results are presented in \cref{figure:circle100_-1_all,figure:pd100_-1}.

The analysis using eigenspace functors can detect such a generator using the eigenspace functor with eigenvalue $-1$.
In other words, prior knowledge about the eigenvalue is essential.
On the other hand, our method can use the same construction both for the inverse mapping and twice mapping.

\begin{figure}[h]
		\centering
		\includegraphics[width=0.75\hsize]{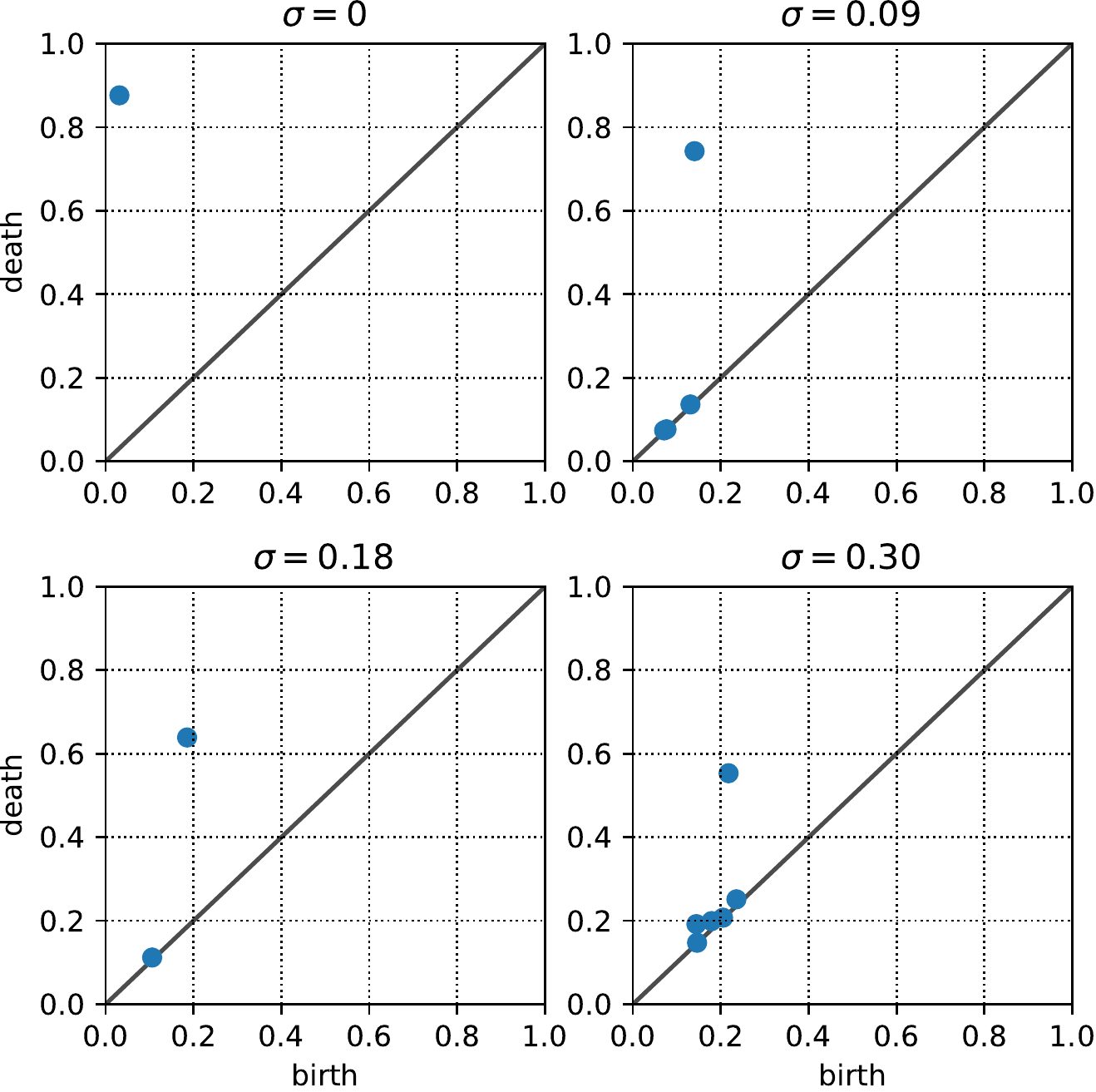}
		\caption{The persistence diagrams for $f(z) = z^{-1}$ with 100 points at $\sigma = 0.00$, $0.09$, $0.18$, and $0.30$.} \label{figure:pd100_-1}
		\includegraphics[width=0.75\hsize]{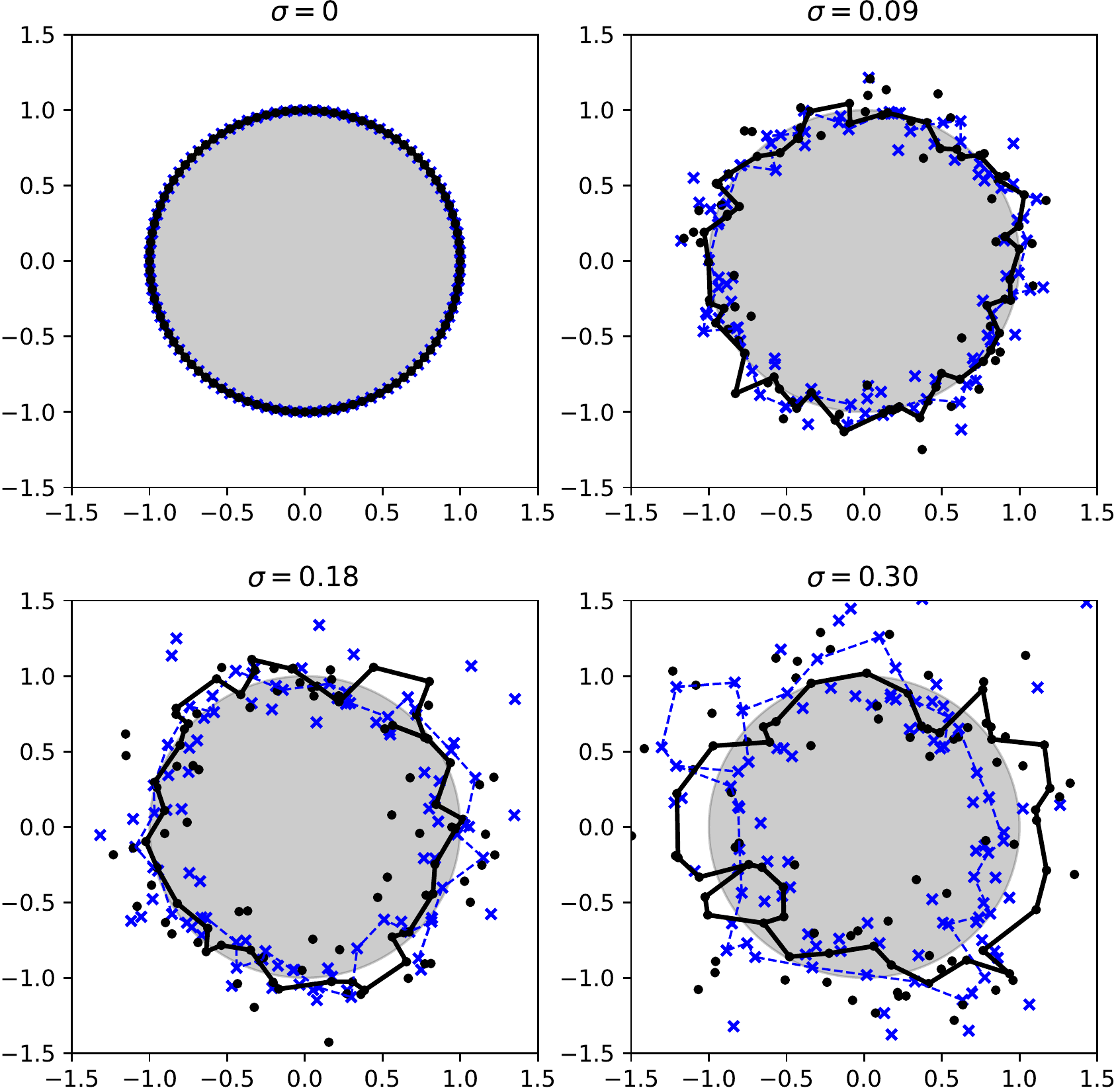}
		\caption{Computational results for $f(z) = z^{-1}$ with 100 points at $\sigma = 0.00$, $0.09$, $0.18$, and $0.30$. The generators correspond to the most persistent birth-death pair in each persistence diagram in \cref{figure:pd100_-1}.} \label{figure:circle100_-1_all}
\end{figure}

\subsection{Mapping on a torus}
Moreover, our method can be applied to maps on the torus $T := \mathbb{R}^2/\mathbb{Z}^2$.
We adopt the metric on $T$ induced by the Euclidean metric on $\mathbb{R}^2$.

Let us consider the self-map on $T$ defined as
\begin{equation}
  A = 
  \begin{pmatrix}
    2 & 1 \\
    1 & 1 \\
  \end{pmatrix}
  \colon T \to T.
\end{equation}
We put 64 sampling points $S := \set{ (\frac{i}{8} , \frac{j}{8}) | \text{$0 \leq i \leq 7$, $0 \leq j \leq 7$}}$ and generate a sampled map of $A$ on $S$.
Such a sampled map of $A$ is a challenging example for the analysis using eigenspace without any prior knowledge because the eigenvalues of $A$ are $\frac{3+\sqrt{5}}{2}$ and $\frac{3-\sqrt{5}}{2}$.

The computational results are presented in \cref{figure:torus_pd,figure:torus_gen_0,figure:torus_gen_0_0,figure:torus_gen_0_1,figure:torus_gen_1,figure:torus_gen_1_0,figure:torus_gen_1_1}.
We remark that the unique point in \cref{figure:torus_pd} has multiplicity $2$.
Let
$  \begin{pmatrix}
    1 \\
    0 \\
  \end{pmatrix}$
and
$  \begin{pmatrix}
    0 \\
    1 \\
  \end{pmatrix}$
be a standard homology basis on the torus.
The generators corresponding to the unique birth-death point are given by $\alpha$ and $505\alpha + \beta$ (=$\frac{1}{2}\alpha + \beta$ with the coefficient $\mathbb{Z}/1009\mathbb{Z}$) in our numerical experiment, where $\alpha$ and $\beta$ are cycles in $HG_b$ at the birth radius $b$ illustrated in \cref{figure:torus_gen_0,figure:torus_gen_0_0,figure:torus_gen_0_1} and \cref{figure:torus_gen_1,figure:torus_gen_1_0,figure:torus_gen_1_1}, respectively.
These cycles correspond to the mappings
\begin{equation}\label{eq:mappings_torus}
  \begin{pmatrix}
    -1 \\
    0 \\
  \end{pmatrix}
  \mapsto
  \begin{pmatrix}
    -2 \\
    -1 \\
  \end{pmatrix}\text{ and }
  \begin{pmatrix}
    0 \\
    1 \\
  \end{pmatrix}
  \mapsto
  \begin{pmatrix}
    1 \\
    1 \\
  \end{pmatrix},
\end{equation}
respectively.
The mappings~\eqref{eq:mappings_torus} are nothing but the map $A= 
  \begin{pmatrix}
    2 & 1 \\
    1 & 1 \\
  \end{pmatrix}
$, concluding the success in reconstructing $A$.

\begin{figure}[h]
  	\centering
  	\includegraphics[width=0.48\hsize]{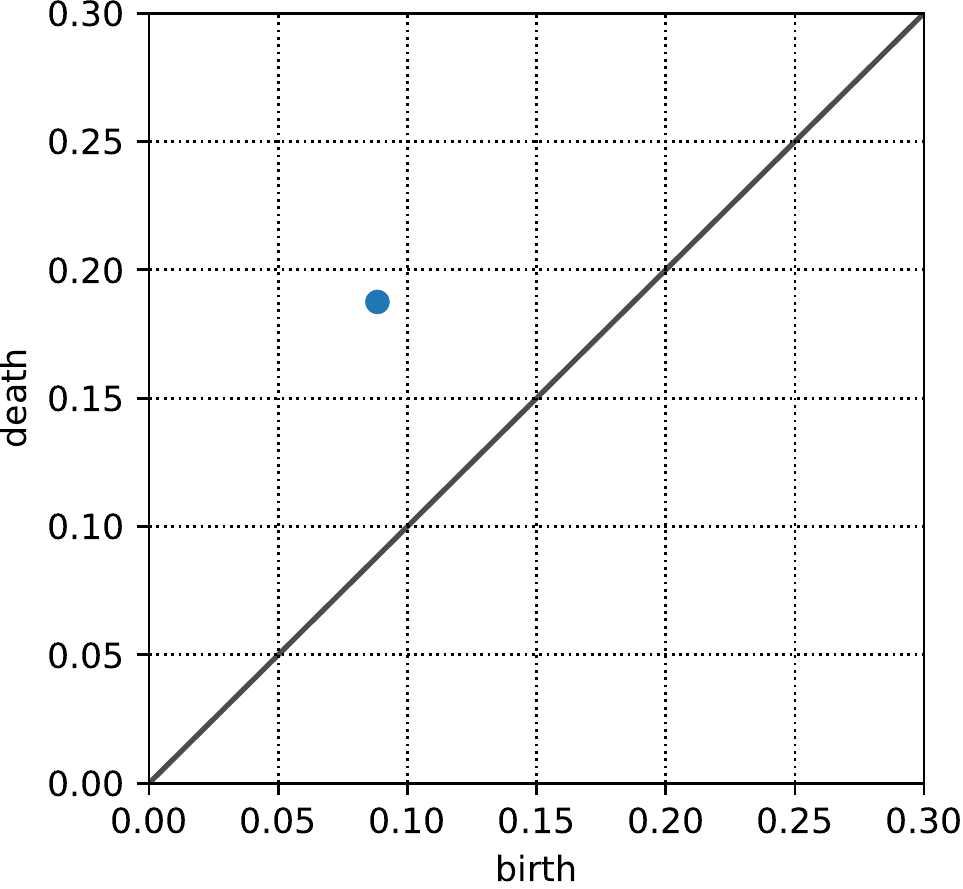}
  	\caption[Text excluding the matrix]{The persistence diagram of a sampled map of $A = \begin{pmatrix} 2 & 1 \\ 1 & 1 \\ \end{pmatrix}$. The unique point has multiplicity 2.} \label{figure:torus_pd}
\end{figure}
\begin{figure}[h]
  \centering
  \begin{minipage}[t]{0.7\hsize}
  	\centering
  	\includegraphics[width=0.6\hsize]{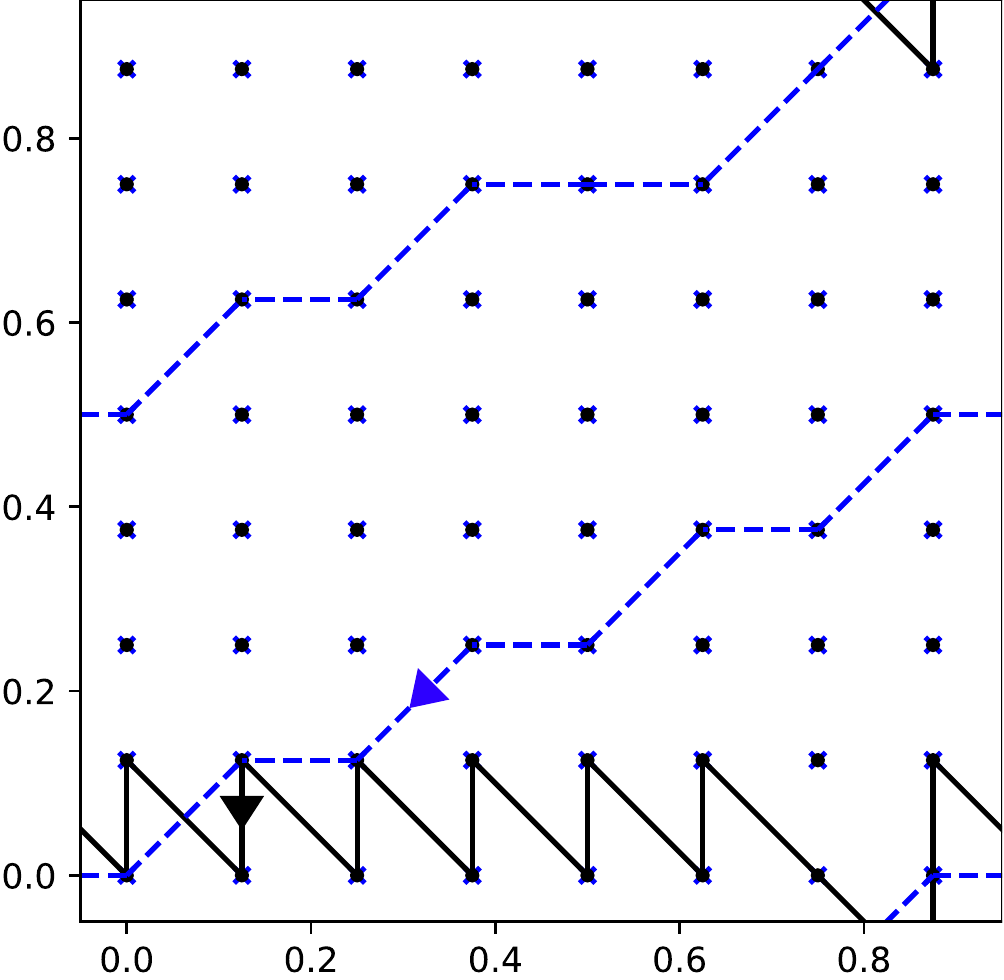}
  	\caption[Text excluding the matrix]{A homology generator $\alpha$ on the torus corresponding to the point in \cref{figure:torus_pd}. The black line is the cycle in domain and mapped to the blue dashed line. Using the standard basis, this mapping is written as $\begin{pmatrix} -1 \\ 0 \\ \end{pmatrix} \mapsto \begin{pmatrix} -2 \\ -1 \\ \end{pmatrix}$.} \label{figure:torus_gen_0}
  \end{minipage}
  \\
  \centering
  \begin{minipage}[t]{0.48\hsize}
  	\centering
  	\includegraphics[width=\hsize]{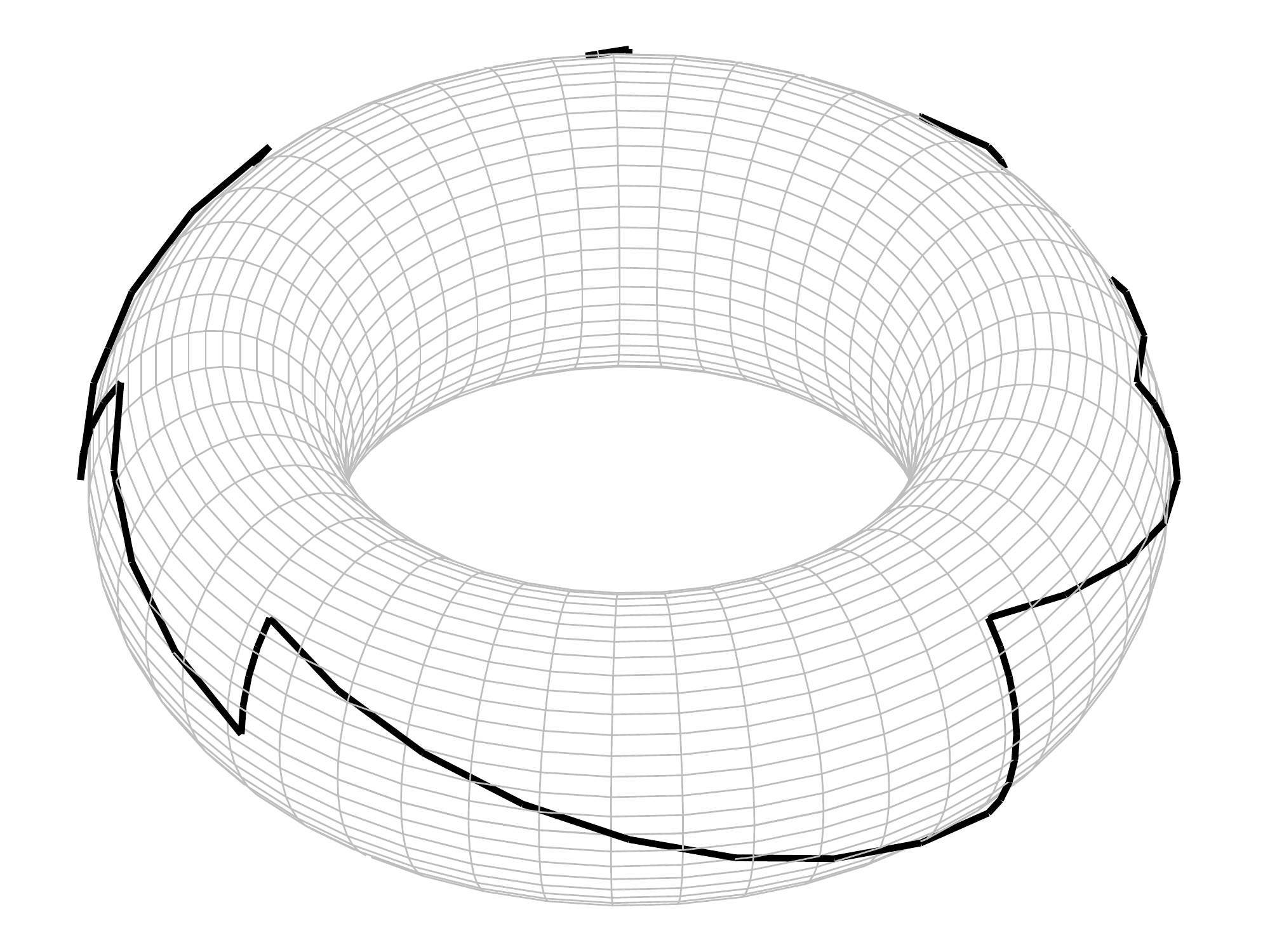}
  	\caption[Text excluding the matrix]{The generator on the domain side embedded in $\mathbb{R}^3$. The loop turns around the torus once in longitudinal direction.} \label{figure:torus_gen_0_0}
  \end{minipage}
  \begin{minipage}[t]{0.02\hsize}
  \end{minipage}
  \begin{minipage}[t]{0.48\hsize}
  	\centering
  	\includegraphics[width=\hsize]{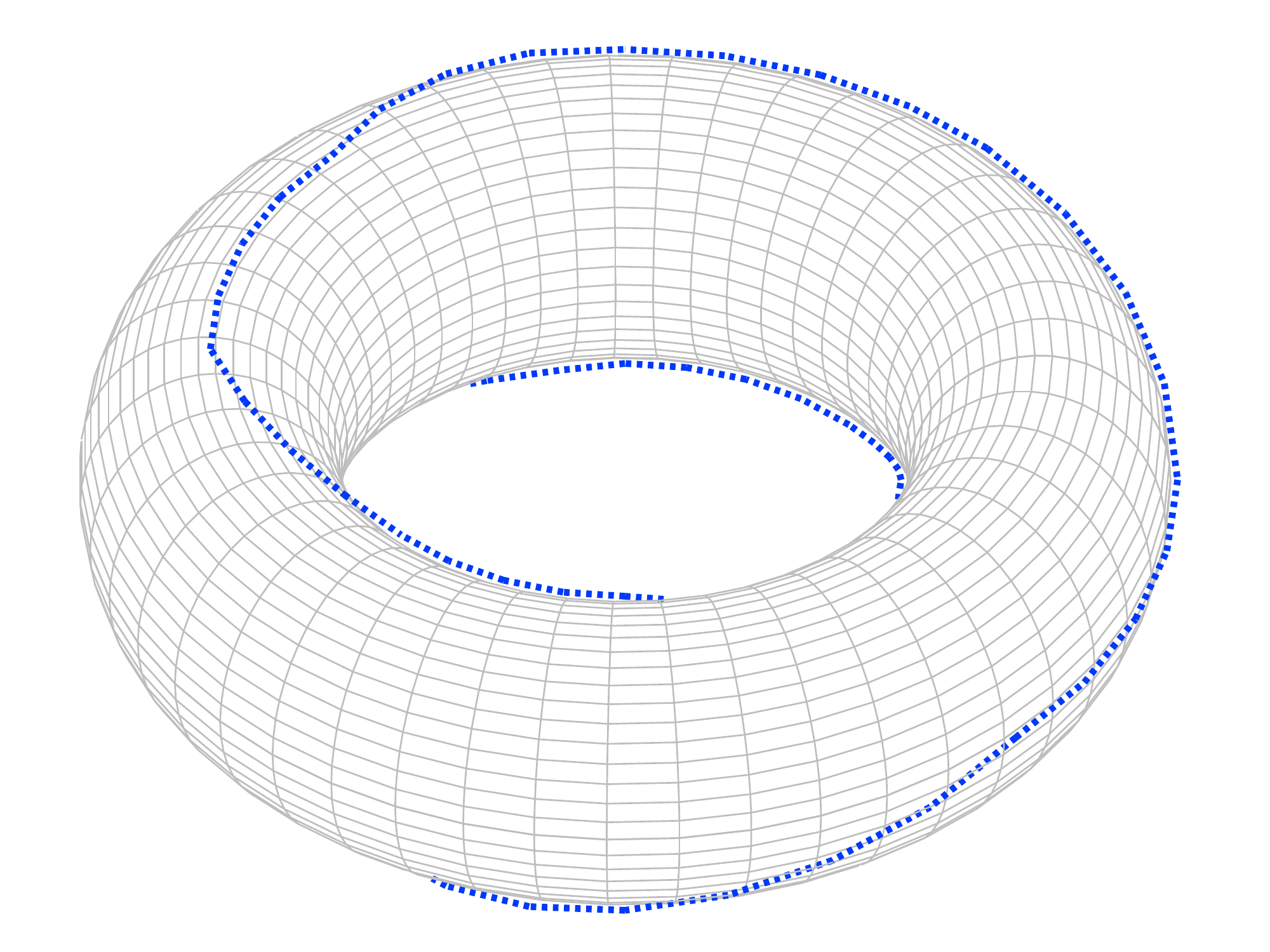}
  	\caption[Text excluding the matrix]{The generator on the image side embedded in $\mathbb{R}^3$. The loop turns around the torus twice in longitudinal and once in meridian direction.} \label{figure:torus_gen_0_1}
  \end{minipage}
\end{figure}
\begin{figure}[h]
  \centering
  \begin{minipage}[t]{0.7\hsize}
  	\centering
  	\vspace{0.05\hsize}
  	\includegraphics[width=0.6\hsize]{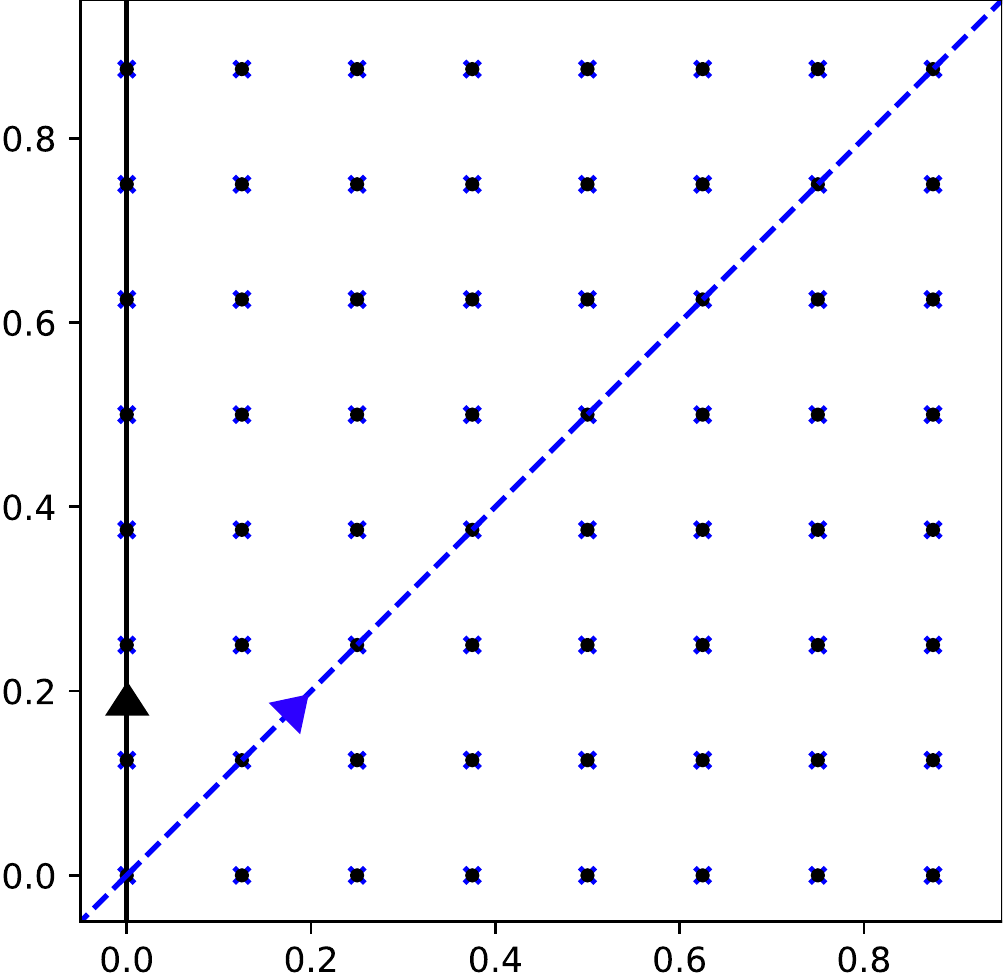}
  	\caption[Text excluding the matrix]{Another generator $\beta$ illustrating $\begin{pmatrix} 0 \\ 1 \\ \end{pmatrix} \mapsto \begin{pmatrix} 1 \\ 1 \\ \end{pmatrix}$. In our numerical experiment, the other generator corresponding to the point in \cref{figure:torus_pd} is given by $505\alpha + \beta$ ($= \frac{1}{2}\alpha + \beta$ with the coefficient $\mathbb{Z}/1009\mathbb{Z}$).}  \label{figure:torus_gen_1}
  \end{minipage}  
  \\
  \centering
  \begin{minipage}[t]{0.48\hsize}
  	\centering
  	\includegraphics[width=\hsize]{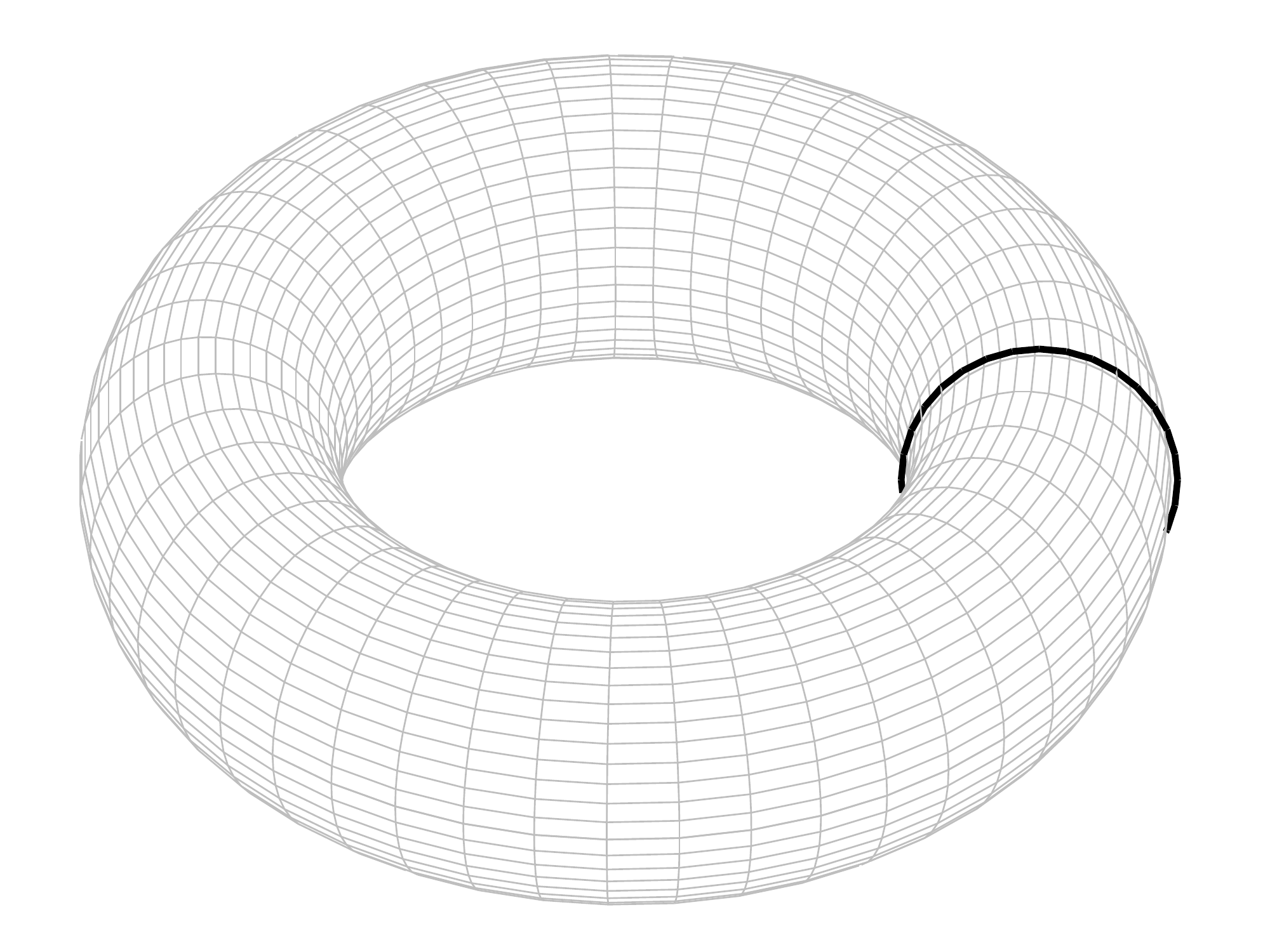}
  	\caption[Text excluding the matrix]{The generator on the domain side embedded in $\mathbb{R}^3$. The loop turns around the torus once in meridian direction.} \label{figure:torus_gen_1_0}
  \end{minipage}
  \begin{minipage}[t]{0.02\hsize}
  \end{minipage}
  \begin{minipage}[t]{0.48\hsize}
  	\centering
  	\includegraphics[width=\hsize]{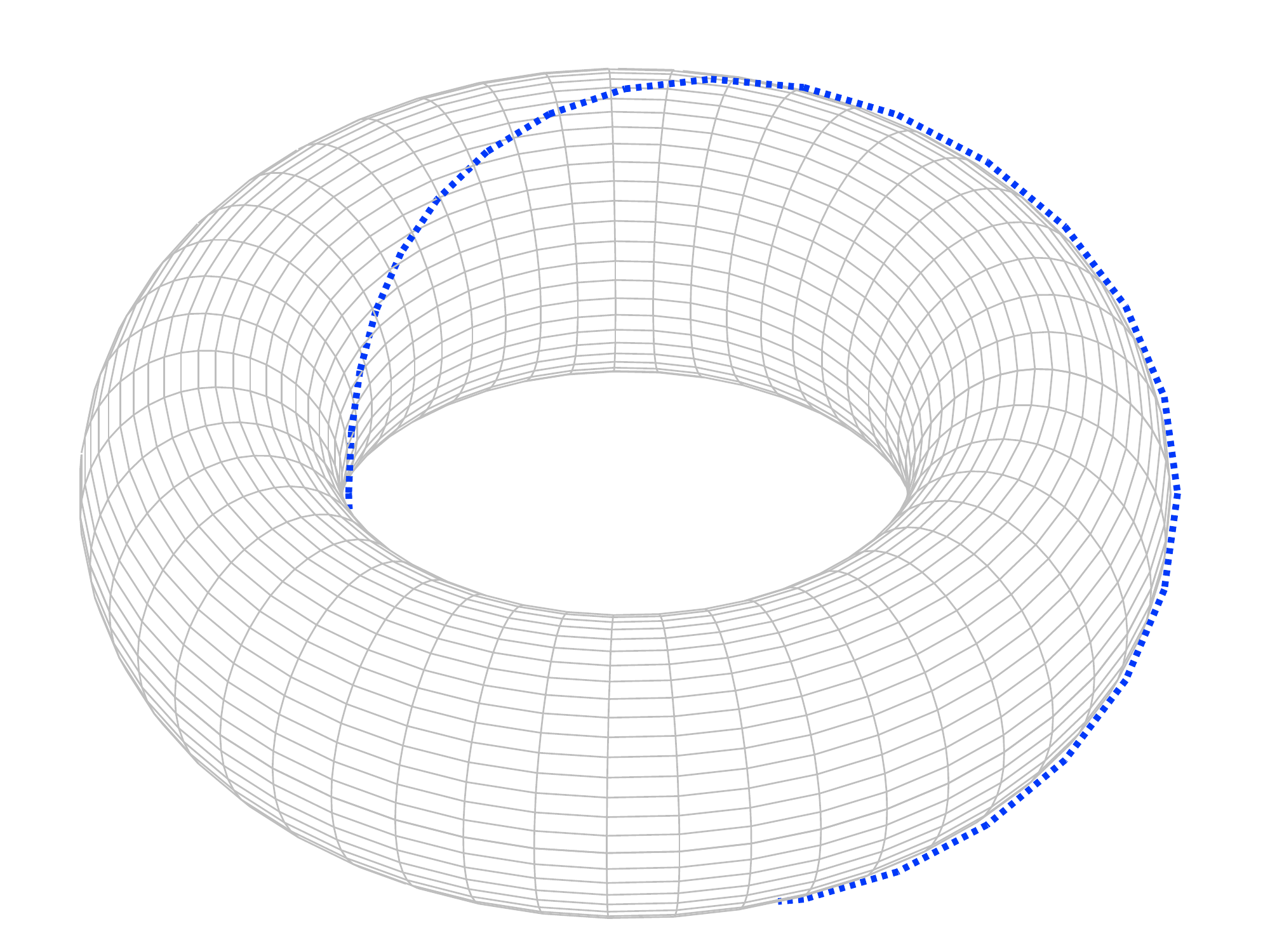}
  	\caption[Text excluding the matrix]{The generator on the image side embedded in $\mathbb{R}^3$. The loop turns around the torus once in longitudinal and once in meridian direction.} \label{figure:torus_gen_1_1}
  \end{minipage}
\end{figure}

\section{Concluding remarks}
In this paper, we defined the persistence diagram of a sampled map and proved that the persistence diagram is uniquely determined and independent of the choice of bases in interval decompositions.
However, the reconstruction of the homology induced map depended on the choice of the bases of interval decomposition.
Our aim is reconstructions of underlying maps, so we must solve the problem on which bases are the best for reconstruction.
This problem will be our future work and would be related to the problem on which cycle is optimal for representing the generator of persistent homology~\cite{optimal_cycle}.

\paragraph{Acknowledgements}
I would like to thank Yasuaki Hiraoka for introducing me to this study, and for many fruitful discussions and stimulating questions.
I would like to thank Emerson G. Escolar for important questions about the well-definedness of the persistence analysis.
I would like to thank Zin Arai for his constructive suggestions and comments on the paper.
I would like to acknowledge the support by Japan Society for the Promotion of Science KAKENHI Grant Number 16J03138, and Japan Science and Technology Agency CREST Mathematics 15656429.


\end{document}